\documentclass[reqno,11pt]{amsart}

\usepackage{graphicx, amsmath, amssymb, amscd, amsthm, euscript, psfrag, amsfonts,bm}

\DeclareMathAlphabet{\mathpzc}{OT1}{pzc}{m}{it}

\newtheorem{thm}[equation]{Theorem}

\newtheorem{rmk}[equation]{Remark}

\newtheorem{prop}[equation]{Proposition}

\newtheorem{cor}[equation]{Corollary}
\newtheorem{lem}[equation]{Lemma}
\newtheorem{dfn}[equation]{Definition}

\numberwithin{equation}{section}

\numberwithin{equation}{section}

\newcommand{\be}{begin{equation}}

\newcommand{\q}{\mathbb{Q}}

\newcommand{\e}{{\varepsilon}}

\renewcommand{\q}{\mathbb{Q}}

\renewcommand{\c}{\mathbb{C}}
\newcommand{\br}{\mathbb{R}}

\newcommand{{\grinv}}{{\Cal G}^{-r}}

\newcommand{\ba}{\backslash}

\newcommand{\G}{\Gamma}

\newcommand{\g}{\gamma}

\newcommand{\Haar}{\operatorname{Haar}}

\newcommand{\Cal}{\mathcal}

\newcommand{\la}{\langle}
\newcommand{\ra}{\rangle}
\newcommand{\sg}{\sigma}

\newcommand{\bp}{\begin{pmatrix}}
\newcommand{\ep}{\end{pmatrix}}

\renewcommand{\bp}{{\rm bp}}

\newcommand{\SO}{\operatorname{SO}}

\newcommand{\T}{\operatorname{T}}

\newcommand{\PSL}{\op{PSL}}

\newcommand{\cS}{{\mathcal S}}

\newcommand{\vcal}{\mathcal{V}}

\newcommand{\op}{\operatorname}

\newcommand{\BR}{\operatorname{BR}}

\newcommand{\BMS}{\operatorname{BMS}}

\renewcommand{\setminus}{-}
\newcommand{\Lie}{{\rm Lie}}

\renewcommand{\be}{\begin{equation}}
\newcommand{\ee}{\end{equation}}

\newcommand{\fG}{\mathfrak B}

\newcommand{\B}{B}

\newcommand{\cusp}{{\rm cusp}}
\newcommand{\hcal}{\mathcal{H}}

\newcommand{\W}{\mathcal W}
\newcommand{\V}{\mathcal V}
\newcommand{\vare}{\varepsilon}

\begin{document}

\title[Holonomy]
{Closed geodesics and holonomies for Kleinian manifolds}

\author{Gregory Margulis}
\address{Mathematics department, Yale university, New Haven, CT 06511}
\email{gregorii.margulis@yale.edu}
\thanks{Margulis was supported in part by NSF Grant \#1265695.}
\author{Amir Mohammadi}
\address{Department of Mathematics, The University of Texas at Austin, Austin, TX 78750}
\email{amir@math.utexas.edu}
\thanks{Mohammadi was supported in part by NSF Grant \#1200388.}

\author{Hee Oh}
\address{Mathematics department, Yale university, New Haven, CT 06511
and Korea Institute for Advanced Study, Seoul, Korea}
\email{hee.oh@yale.edu}
\thanks{Oh was supported in part by NSF Grant \#1068094.}






\begin{abstract} For a rank one Lie group $G$ and a Zariski dense and
 geometrically finite subgroup $\G$ of $G$, we establish 
the joint equidistribution of  closed geodesics and their holonomy classes
 for the associated locally symmetric space.
Our result  is given in a quantitative form for geometrically finite real hyperbolic manifolds  whose 
critical exponents  are big enough. In the case when $G=\PSL_2(\c)$, our results imply the
equidistribution of eigenvalues of elements of $\G$ in the complex plane.
 
When $\G$ is a lattice, the equidistribution of holonomies
was proved by Sarnak and Wakayama in 1999 using the Selberg trace formula. 

\end{abstract}

\maketitle

\section{Introduction}
A rank one locally symmetric space $X$  is of the form $\Gamma\ba G/ K$ where $G$ is a connected simple
linear Lie group of real rank one, $K$ is a maximal compact subgroup of $G$ and $\Gamma$ is a torsion-free discrete subgroup
 of $G$.  Let $o:=[K]\in G/K$ and choose a unit tangent vector $v_o$ at $o$. Let $M$ denote the
 subgroup of $G$  which stabilizes $v_o$. The unit tangent bundle $\T^1(X)$ of $X$ can be identified with
 $\G\ba G/M$.
 Each closed geodesic $C$ on $\T^1(X)$ gives rise to the holonomy conjugacy class $h_C$ in $M$ which is obtained 
by parallel transport about $C$.

Our aim in this paper is to establish
the equidistribution of holonomies about closed geodesics $C$ with length $\ell(C)$
going to infinity, when $\G$ is geometrically finite and
Zariski dense in $G$.  We will indeed prove a stronger joint equidistribution theorem for closed geodesics and their holonomy classes.
A discrete subgroup $\G$ is called
  {\it geometrically finite} if the unit neighborhood of its convex core in $X$ is of finite Riemannian volume (cf. \cite{Bow}).
 Lattices are clearly geometrically finite, but there is also a big class
 of discrete subgroups of infinite co-volume which are geometrically finite.
 For instance, if $G/K$ is the real hyperbolic space $\mathbb H^n$ and
 $G=\op{SO}(n,1)^\circ$ is the group of its orientation preserving isometries, any discrete group $\G$ admitting a finite-sided
 convex fundamental domain is geometrically finite. The fundamental group
 of a finite volume hyperbolic manifold with non-empty totally geodesic boundary is also known to be geometrically finite. We denote by
 $\delta=\delta_\G$ the critical exponent of $\G$. It is well-known that $\delta>0$ if $\G$ is non-elementary.

In this paper,  a closed geodesic in $\T^1(X)$ is always meant to be a {\it primitive} closed geodesic, unless mentioned otherwise.
For $T>0$, we set
 $$\mathcal G_\G(T):=\{C: \text{$C$ is a  closed geodesic in $\T^1(X)$,}\; \ell(C)\le T\}.$$
 

The following theorem follows from a stronger joint equidistribution theorem \ref{eqt}.
  \begin{thm}\label{m1} Let $\G$ be geometrically finite and Zariski dense in $G$. 
   Then for any continuous class function $\varphi$ on $M$,
 $$\sum_{C\in \mathcal G_\G(T)} \varphi(h_C) \sim
 \frac{e^{\delta T}}{\delta T} \int_M \varphi \; dm\quad\text{as $T\to \infty$} $$
 where $dm$ is the Haar probability measure on $M$.
 \end{thm}
       The  asymptotic of $\# \mathcal G_\G(T)$ was well-known, due to
Margulis \cite{M_T} for $X$ compact, to Gangolli and Warner \cite{GW} for $X$ noncompact but of finite volume, and to Roblin \cite{R_T} for $X$ geometrically finite:
$$\# \mathcal G_\G(T)\sim \frac{e^{\delta T}}{{\delta} T}.$$
 We do not rely on this result in our proof of Theorem \ref{m1}.

If we define $ \mathcal G^{\dag}_\G(T)$ to be the set of all primitive and non-primitive closed geodesics of length at most $T$,
then it is easy to see 
that  $$ \# \mathcal G^{\dag}_\G(T)=  \# \mathcal G_\G(T)+ O(T) \cdot  \# \mathcal G_\G(T/2) .$$

Therefore Theorem \ref{m1} remains the same if we replace $\mathcal G_\G(T)$ by
$ \mathcal G^{\dag}_\G(T)$.
It is worth mentioning that if one considers all geodesics, then 
it follows from the work of Prasad and Rapinchuk \cite{PR} that
the set of all holonomy classes about closed geodesics in $\T^1(X)$ is dense in the space of  all conjugacy classes of $M$.

When $G=\SO(n,1)^\circ$ and
 $\G$ is a co-compact lattice, Theorem \ref{m1}  was known due to  Parry and Pollicott  \cite{PP}, who showed that the topological mixing
 of the frame flow on a compact manifold implies the equidistribution of holonomies.
 When $\G$ is a lattice in a general rank one group $G$, Theorem \ref{m1}
  was proved by Sarnak and Wakayama \cite{SW}; their method is based on the Selberg trace formula
  and produces an error term.  Therefore Theorem \ref{m1} is new only
   when $\G$ is of infinite co-volume in $G$. However our approach gives a more direct
  dynamical proof of Theorem \ref{m1} even in the lattice case.

Recall the log integral function:
  $$\op{li}(x):=\int_{2}^{x}  \frac{dt}{\log t} \sim \frac{x}{\log x} \left[1+\frac{1}{\log x} + \cdots \right] .$$ 
\begin{thm}\label{m2}
Let $\G$ be a geometrically finite subgroup of $G=\SO(n,1)^\circ$ with $n\ge 3$.
We suppose that  $\delta> \max\{n-2, (n-2+\kappa)/2\} $ where $\kappa$
is the maximum rank of all parabolic fixed points of
$\G$. Then there exists $\eta>0$ such that
 for any smooth class function $\varphi$ on $M$,
 $$\sum_{C\in \mathcal G_\G(T)} \varphi(h_C) =
\op{li} (e^{\delta T})   \int_M \varphi \; dm +O(e^{(\delta -\eta) T}) \quad\text{as $T\to \infty$} $$
 where   the implied constant depends only on the Sobolev norm of $\varphi$.
\end{thm}
 Theorem \ref{m2} gives a quantitative counting result for closed geodesics of length at most $T$.
This was known 
when $\G$ is a lattice by the work of Selberg, and Gangolli-Warner \cite{GW} by the trace formula approach,
or when $\G$ is a convex co-compact subgroup of $\SO(2,1)$ by Naud \cite{Na} by the symbolic dynamics approach.
 We remark that Theorem \ref{m2} can be extended to geometrically finite groups in other rank one Lie groups for which
Theorem \ref{hem} holds; this will be evident from our proof.

As is well-known, the set of closed geodesics in $\T^1(X)$ is in one-to-one correspondence with
the set of conjugacy classes of primitive hyperbolic elements of $\G$.
If $A=\{a_t\}$ denotes the one parameter subgroup whose right translation action on $\G\ba G/M$
corresponds to the geodesic flow on $\T^1(X)$, then any hyperbolic element $g\in G$ is conjugate to
$a_g m_g$ with $a_g\in A^+:=\{a_t: t>0\}$ and $m_g\in M$. Moreover $a_g$ is uniquely determined,
and $m_g$ is uniquely determined up to a conjugation in $M$. 
Denote by $[\gamma]$ the conjugacy class of $\gamma$ in $\Gamma$ and
by $[\Gamma_{ph}]$ the set of all conjugacy classes of {\it primitive hyperbolic} elements of $\Gamma$.
Given a closed geodesic $C$  in $\T^1(X)$, if $[\gamma]\in [\Gamma_{ph}]$ is the corresponding conjugacy class,
 then the holonomy class $h_C$ is precisely the conjugacy class 
$[m_\gamma]$. Therefore Theorem \ref{m1}
can also be interpreted as
the equidistribution of $[m_\gamma]$'s among
primitive hyperbolic conjugacy classes of $\Gamma$.

For $G=\PSL_2(\c)$,
Theorem \ref{m1} implies the equidistribution of  eigenvalues of $\Gamma$.
If we denote by $\lambda_\gamma$ and $ \lambda_\gamma^{-1}$ the eigenvalues
of $\gamma \in \G$ (up to sign) so that  $|\lambda_\gamma|\ge 1$,
then $\gamma$
is hyperbolic if and only if $|\lambda_\gamma|>1$.

The aforementioned result of Prasad and Rapinchuk says that any Zariski dense subgroup $\Gamma$ contains a hyperbolic element
$\gamma$ such that the argument of the complex number $\lambda_\gamma$ is an irrational multiple of $\pi$ \cite{PR}.
We show a stronger theorem that the arguments of $\lambda_\gamma$'s are equidistributed
in all directions when $\Gamma$ is geometrically finite.
\begin{thm}\label{m3}
Let $G=\PSL_2(\c)$ and $\G$ be a geometrically finite and Zariski dense subgroup of $G$.
 
 For any $0<\theta_1<\theta_2<\pi$, we have
 \be \label{pp} \#  \{[\gamma]\in [\Gamma_{ph}]:  |\lambda_\gamma|<T,\; \theta_1< \op{Arg}(\lambda_\gamma )<\theta_2\} \sim 
 \frac{ (\theta_2-\theta_1)T^{2\delta}}{2\pi \delta \log T} \quad \text{as $T\to \infty$.} 
  \ee 
   If $\delta >1$ and $\Gamma$ has no rank $2$ cusp,  or if  $\delta>3/2$ in general, 
   then there exists $\e_0>0$ such that
    \be \label{pp2} \#  \{[\gamma]\in [\Gamma_{ph}]:  |\lambda_\gamma|<T,\; \theta_1< \op{Arg}(\lambda_\gamma )<\theta_2\} =
 \frac{ (\theta_2-\theta_1)}{\pi} \op{li} (T^{2\delta}) + O(T^{2\delta -\e_0}) .\ee
     \end{thm}

For a hyperbolic element $\gamma\in \G$, the length
 of the corresponding geodesic 
 is $2\log |\lambda_\gamma |$ and the argument of $\lambda_\gamma$ is precisely the holonomy 
 associated to $\gamma$. 
Hence Theorem \ref{m3} is a special case of Theorems \ref{m1} and \ref{m2}.


   In the case when $\Gamma$ is contained in an arithmetic subgroup of $\PSL_2(\c)$,
  the polynomial error term can be taken to be uniform over all
congruence subgroups of $\Gamma$; this follows from our approach based on the work
of Bourgain, Gamburd and Sarnak \cite{BGS} and of Mohammadi and Oh \cite{MO}. In the case when $\Gamma\subset
\PSL_2(\mathcal O_D)$ where $\mathcal O_D$ is the ring of integers of an imaginary quadratic extension
 $\q(\sqrt{-D})$ of $\q$, the eigenvalues
of $\G$ are fundamental units of $\mathcal O_D$ (cf. \cite{Sa}), in which case
Theorem \ref{m3} also bears an arithmetic application on the distribution of such fundamental units arising from $\G$.

\medskip

  In proving Theorem \ref{m1}, we consider the following measure $\mu_T$ on the product space $\T^1(X) \times M^{\textsc{c}}$
where $M^{\textsc{c}}$ denotes the space of
conjugacy classes of $M$: for $f\in C(\T^1(X))$ and $\xi\in C(M^{\textsc{c}})$, set
\be \label{deee}\eta_T(f\otimes \xi):=\sum_{C\in  \mathcal G_\G(T)} \mathcal{D}_C (f) \xi(h_C)\ee
where  $\mathcal{D}_C$ denotes the length measure on the geodesic $C$, normalized to be a probability measure.
Theorem \ref{m1} follows if we show that
for any {\it bounded} continuous function $f$ and a continuous function $\xi$,
\be \label{mf} \eta_T(f\otimes \xi) \sim \frac{e^{\delta T}\cdot  m^{\BMS}(f) \cdot \int_M \xi dm}{\delta\cdot |m^{\BMS}|\cdot T } 
\quad\text{as $T\to \infty$}\ee
where $m^{\BMS}$ is the Bowen-Margulis-Sullivan measure
on $\T^1(X)$. 
 
We will deduce \eqref{mf} from the following:
\be \label{mff} \mu_T(f\otimes \xi) \sim \frac{e^{\delta T} \cdot m^{\BMS}(f) \cdot \int_M \xi dm }{\delta\cdot |m^{\BMS}| } 
\quad\text{as $T\to \infty$}\ee
where $\mu_T(f\otimes \xi)=\sum_{C\in  \mathcal G_\G(T)} \mathcal{L}_C (f) \xi(h_C)$
for  $\mathcal{L}_C=\ell(C) \cdot \mathcal{D}_C$ (the length measure on $C$).

Let $N^+$ and $N^-$ denote the expanding and contracting horospherical subgroups
of $G$ with respect to $A$, respectively.
 In studying \eqref{mff}, the following $\e$-flow boxes play an important role: for $g_0\in G$, set
\be\label{eq:G-ee}
\mathfrak{B}(g_0,\e)=g_0 (N^+_\e N^-\cap N^-_\e N^+AM) M_\e A_\e .
\ee where $A_\e$ (resp. $M_\e$) is the $\e$-neighborhood of $e$ in $A$ (resp. $M$)
and $N^{\pm}_\e$ denotes the $\e$-neighborhood of $e$ in $N^{\pm}$. 
Let $\tilde \fG(g_0,\e)$ denote
the image of $\mathfrak{B}(g_0,\e)$ under the canonical projection $G\to \G\ba G/M$.
Fixing a Borel subset $\Omega$ of $M$ which is conjugation-invariant,
the main idea is to relate the restriction of $\mu_T$ to  $\tilde \fG(g_0,\e) \otimes \Omega$
with the counting function of the set $\G\cap  \fG(g_0,\e)A_T^+\Omega \fG(g_0,\e)^{-1}$ with 
$A_T^+=\{a_t:0<t\le T\}$ (see Comparison lemma \ref{comp}); 
we establish this relation using the effective closing lemma \ref{eecl}. We remark that for the effective closing lemma, it is quite essential to
use a flow box which is precisely of the form given in \eqref{eq:G-ee}. This flow box was
first used in Margulis' work on counting closed geodesics \cite{M_T}.  The counting function
 of $\G\cap  \fG(g_0,\e)A_T^+\Omega \fG(g_0,\e)^{-1}$  can then be understood 
based on the mixing result of Winter \cite{Wi}, which says that the $A$ action on $L^2(\G\ba G, m^{\BMS})$ is mixing.
   An effective mixing statement for the cases mentioned in Theorem \ref{m2} was obtained in \cite{MO}. We also remark that if
   we restrict ourselves only to those $f$ with compact support, then
   \eqref{mff} holds for any discrete subgroup $\G$ admitting a finite BMS measure; that is, $\G$ need not be geometrically finite.

\medskip

\noindent{\bf Acknowledgement}
We would like to thank Dale Winter for helpful comments on the preprint.  We also thank the referee for the careful reading of our manuscript
and helpful comments.
\section{Preliminaries}\label{sec:notation}
Throughout the paper, let $G$ be a connected simple real linear Lie group of real rank one. 
As is well known, $G$
is one of the following type: $\SO(n,1)^\circ$, $\op{SU}(n,1)$, $\op{Sp}(n,1)$ ($n\ge 2$) and $\op{F}_{4}^{-20}$,
which are the groups of isometries of the hyperbolic spaces
 $\mathbb H_{\br}^n$,  $\mathbb H_{\c}^n$, $\mathbb H_{\mathbb H}^n$, $\mathbb H_{\mathbb O}^2$ respectively.
 Let $K$ be a maximal compact subgroup of $G$.
Then $\tilde X:=G/K$ is a symmetric space of rank one. Let $o\in \tilde X$ be the point which is stabilized by $K$.
  The killing form on the Lie algebra of $G$
endows a left $G$-invariant metric $d_{\tilde X}$ on $\tilde X$ which we normalize so that the maximum sectional curvature is $-1$.

The volume entropy $D(\tilde X)$ of $\tilde X$ is defined by
\be\label{ve1}D(\tilde X)=\lim_{T\to \infty}\frac{\log \text{Vol}(B(o, T))}{T}\ee
where $B(o, T)=\{x\in \tilde X: d_{\tilde X}(o, x)\le T\}$.
It is explicitly given as follows:
\be \label{ve2} D(\tilde X)= n-1,2n, 4n+2, 22\ee respectively
for $\SO(n,1)^\circ$, $\op{SU}(n,1)$, $\op{Sp}(n,1)$ and $\op{F}_{4}^{-20}$.
 
We denote by $\partial_\infty(\tilde X)$ the geometric boundary of $\tilde X$ and by $\T^1(\tilde X)$ the unit tangent bundle
of $\tilde X$. Fixing
a vector $v_o\in \T^1(\tilde X)$ based at $o$,   $\T^1(\tilde X)$ 
 can be identified with $G/M$ where $M$ is the stabilizer of $v_o$ in $G$.
 For  a vector $v\in \T^1(\tilde X)$, we denote by $v^+ \in\partial_\infty(\tilde X) $ and $ v^-\in\partial_\infty(\tilde X)$
 the forward and the backward end points of the geodesic determined by $v$.  For $g\in G$,
 we set $g^{\pm}=(gv_o)^{\pm}$.
There exists  a one parameter subgroup $A=\{a_t: t\in \br\}$ of diagonalizable elements of $G$
which commutes with $M$ and whose right translation action on $G/M$ by $a_t$ corresponds to the geodesic flow
for time $t$ on $\T^1(\tilde X)$; in fact, $M$ is equal to the centralizer of $A$ in $K$.
We set $$A^+:=\{a_t: t>0\}\quad \text{ and }\quad A_T^+:=\{a_t: 0<t\le T\}.$$
We denote by $N^+$ and $N^-$ the expanding and contracting horospherical subgroups:
$$
N^+=\{g\in G: \mbox{$\lim_{t\to +\infty}a_{t}ga_{-t}\to e$}\};$$
$$N^-=\{g\in G: \mbox{$\lim_{t\to +\infty}a_{-t}ga_{t}\to e$}\}. $$

The stabilizer of $v_o^+$ and $v_o^-$ in $G$
 are given respectively by $$P^-:=MAN^-,\quad\text{and}\quad P^+:=MAN^+ .$$
 Hence the orbit map $g\mapsto gv_o^{+}$ (resp. $g\mapsto gv_o^-$) induces a homeomorphism
 between $G/P^-$ (resp. $G/P^+$) with $\partial_\infty(\tilde X)$.  


Let $d=d_G$ be a left $G$-invariant Riemannian metric on $G$ which induces the
metric $d_{\tilde X}$ on $\tilde X=G/K$.
  For a subset $S$ of $G$ and $g_0>0$, we set 
$$S_\e(g_0):=\{s\in S: d_G(g_0, s)<\e\}$$ the intersection of
the $\e$-ball at $g_0$ with $S$.
Hence the $\e$-balls $G_\e(g_0)$ form a basis of open neighborhoods at $g_0$.


\medskip

\noindent{\bf Flow box:}
Following Margulis~\cite{M_T}, we will
define the flow-box around $g_0\in G$ for all small $\e>0$. For this, we will use the following
$\e$-neighborhoods of $e$ in $N^+, N^-, A, M$.

 The groups $N^{\pm}$ are connected unipotent groups and hence the exponential map $\exp: \Lie(N^{\pm})\to N^{\pm}$ is
 a diffeomorphism. For $\e>0$, we set
  $$N^{\pm}_{\e}:= \{n_x^{\pm}:=\exp x \in N^{\pm}:  \|x\|<\e\}$$
  where $\|x\|$ denotes a norm on the real vector space $\Lie(N^{\pm})$ which is
  $M$-invariant under the adjoint action of $M$ on $\Lie(N^{\pm})$.
  
 For $A$ and $M$, we simply put $$A_\e= A\cap G_\e(e)=\{a_t: t\in (-\e, \e)\}, \text{ and } M_\e=M\cap
 G_\e(e).$$

We now define the $\e$-flow box $\mathfrak{B}(g_0,\e)$ at $g_0$ as follows:
\be\label{eq:G-e}
\mathfrak{B}(g_0,\e)=g_0 (N^+_\e N^-\cap N^-_\e N^+AM) M_\e A_\e .
\ee
For simplicity, we set $ \mathfrak{B}(\e):=\mathfrak{B}(e,\e)$.
The product maps $N^+\times A\times M\times N^- \to G$ and
  $N^-\times A\times M\times N^+ \to G$ are diffeomorphisms 
 onto Zariski open neighborhoods of $e$ in $G$. Therefore  
  the sets $\fG(g_0,\e)$, $\e>0$ form a basis of neighborhoods of $g_0$ in $G$.
  
We remark that
this definition of the flow box is quite essential in our proof of the effective closing lemma \ref{eecl}.
We list the following properties of the flow box which we will use later:
\begin{lem}[Basic properties of the flow box]\label{box} 
Let $g_0\in G$ and $\e>0$.
\begin{enumerate}
\item For any $g\in \mathfrak{B}(g_0,\e)$,
the set $\{t\in \br: ga_t \in  \mathfrak{B}(g_0,\e)\}$ is of Lebesgue length $2\e$;
\item $\mathfrak{B}(g_0,\e)v_o^+ =g_0 N_\e^- v_o^+$ and 
$\mathfrak{B}(g_0,\e)v_o^- =g_0 N_\e^+ v_o^-$;
\item There exists $c>1$ such that 
\be\label{comp'}
G_{c^{-1} \e} (g_0) \subset \mathfrak B(g_0, \e)\subset G_{c \e}(g_0) ; \ee
here $c$ is independent of $g_0\in G$ and all small $\e>0$.
\end{enumerate}
\end{lem}

Considering the action of $g\in G$ on the compactification $\tilde X\cup \partial(\tilde X)$, $g$ is called
{\it elliptic}, {\it parabolic}, {\it hyperbolic} if
the set $\text{Fix}(g)=\{x\in\tilde X\cup \partial(\tilde X): g(x)=x\}$
 of fixed points by $g$ is contained in $\tilde X$, is a singleton on $\partial(\tilde X)$, and
consists of two distinct points on $\partial(\tilde X)$ respectively.
Any element $g$ in a rank one Lie group is one of these three types.

Equivalently, $g\in G$ is {\it elliptic} if $g$ is conjugate to an element of $K$, and
{\it parabolic} if $g$ is conjugate to an element of $MN^+ -M$, and {\it hyperbolic}
if $g$ is conjugate to an element of $A^+M-M$.

 \begin{lem} \label{uq} Suppose that for some $h\in G$,
  $ha_1 m_1 h^{-1}= a_2 m_2$ with $a_1,a_2\in A^+$ and $m_1, m_2\in M$.
  Then $a_1=a_2$, $m_1=m m_2m^{-1}$ for some $m\in M$ and $h\in AM$.
  \end{lem}  
  \begin{proof}
For $g\in G$, define
$$N^\pm(g)=\{q\in G: g^{\ell }q g^{-\ell }\to e\text{ as $\ell \to \pm \infty$}\}.$$
Putting $g_i=a_im_i\in A^+M$ for $i=1,2$, we have
$N^\pm(g_i)=N^\pm.$ On the other hand, since $g_2=hg_1h^{-1}$,
the above definition implies $N^\pm(g_2)=hN^\pm(g_1)h^{-1}.$
Hence $h$ belongs to the common normalizer of $N^{\pm}$, which
is equal to $P^+\cap P^-=AM$. Therefore $h=am\in AM$.
It now follows that $h a_1m_1h^{-1}=a_1 (mm_1m^{-1}) =a_2 m_2$. Hence $a_2^{-1}a_1\in A\cap M=\{e\}$; so
 $a_1=a_2$, as well as $m_2=mm_1m^{-1}$. \end{proof}

As an immediate corollary, we have:
\begin{cor} \label{cuq}  If  a hyperbolic element
 $g\in G$ is  of the form:
 \be\label{amnot}  g=h_g a_gm_g h_g^{-1}\ee with  $a_g m_g\in A^+M$, then $a_g$ is
  uniquely determined, $m_g\in M$  is determined unique up to conjugation and
  $R_g:=h_g Av_o$ is independent of the choice of $h_g$. \end{cor}
  
The geodesic $R_g:=h_gAv_o\subset \tilde X$ is called {\it the oriented axis} of $g$:
$g$ preserves $R_g$,
 and acts as a translation by $T:=d(a_g, e)$. 
  \medskip

\medskip

Let $\Gamma$ be a torsion-free and non-elementary discrete subgroup of $G$.
A closed geodesic  $C$ of length $T>0$ on $\T^1( X)=\Gamma\ba G/M$
is a compact set of the form $\G\ba \G g A M/M$ for some $g\in G$ such that
$gAMg^{-1}\cap \Gamma $ is generated by a hyperbolic element $\gamma=g a_\gamma m_\gamma g^{-1}$
with $T=d(a_\gamma, e)$.  The conjugacy class $[m_\gamma]$ in $M$ is called the holonomy class attached to $C$.
Note that if we have $$\G\ba \G g m_0 a_T  =\G \ba \G g m_0 m$$ for some $m_0,m\in M$, then
 $[m]=[m_\gamma]$.
Geometrically, $\G\ba \G g  m_0 $ is a frame which contains 
the tangent vector $\G\ba \G g M$, 
and the element $m$
measures the extent to which parallel transport around the closed geodesic
$\G g m_0a_T$ differs from the original frame $\G \ba \G g m_0$.
If we choose a different base point $m_1$ from $m_0$, then $m$ changes by
a conjugation; hence the holonomy class attached to $C$ is well-defined.

\section{Effective closing lemma}\label{sec:closing-lemma}
Let $\G$ be a torsion-free, non-elementary and  discrete subgroup of $G$ and 
set $X:=\G\ba G/K$, which is a rank one locally symmetric manifold whose fundamental group
is isomorphic to $\Gamma$. We denote by $\pi:
G\to\G\ba G$ the canonical projection map.

 For two elements $h_1, h_2\in G$, we will write $h_1\sim_\e h_2$
 if  $d_G(h_1, h_2)\le  \e$ and $h_1\sim_{O(\e)} h_2 $ if $d_G(h_1, h_2)\le c \e$ for
 some constant $c>1$ depending only on $G$. For conjugacy classes $[m_1]$ and $[m_2]$ in $M$,
 we write  $[m_1]\sim_{\e} [m_2]$ and
  $[m_1]\sim_{O(\e)} [m_2]$ if, respectively, $m_1\sim_{\e} m_2$ and
  $m_1\sim_{O(\e)} m_2$ for some representatives $m_i\in M$ of $[m_i]$.

 For a subset $S$ of $G$ and $\e>0$,
we also use the notation $G_{O(\e)}(S)$ for the $c\e$-neighborhood of $S$ for some $c>1$ 
depending only on  $G$, and the notation $1_S$ for the characteristic function of $S$.

For $g_0\in G$,
we will define the injectivity radius of $g_0$ in $\G\ba G$ 
to be the supremum $\e>0$ such that the $\e$ flow box $\fG (g_0,\e)$ injects to $\G\ba G$. In what follows, we will consider boxes $\fG (g_0,\e)$ only for those $\e$ which are smaller than the injectivity radius of $g_0$,
without repeatedly saying so.

\medskip

\begin{figure}\label{closing}
  \begin{center}
   \includegraphics[width=1in]{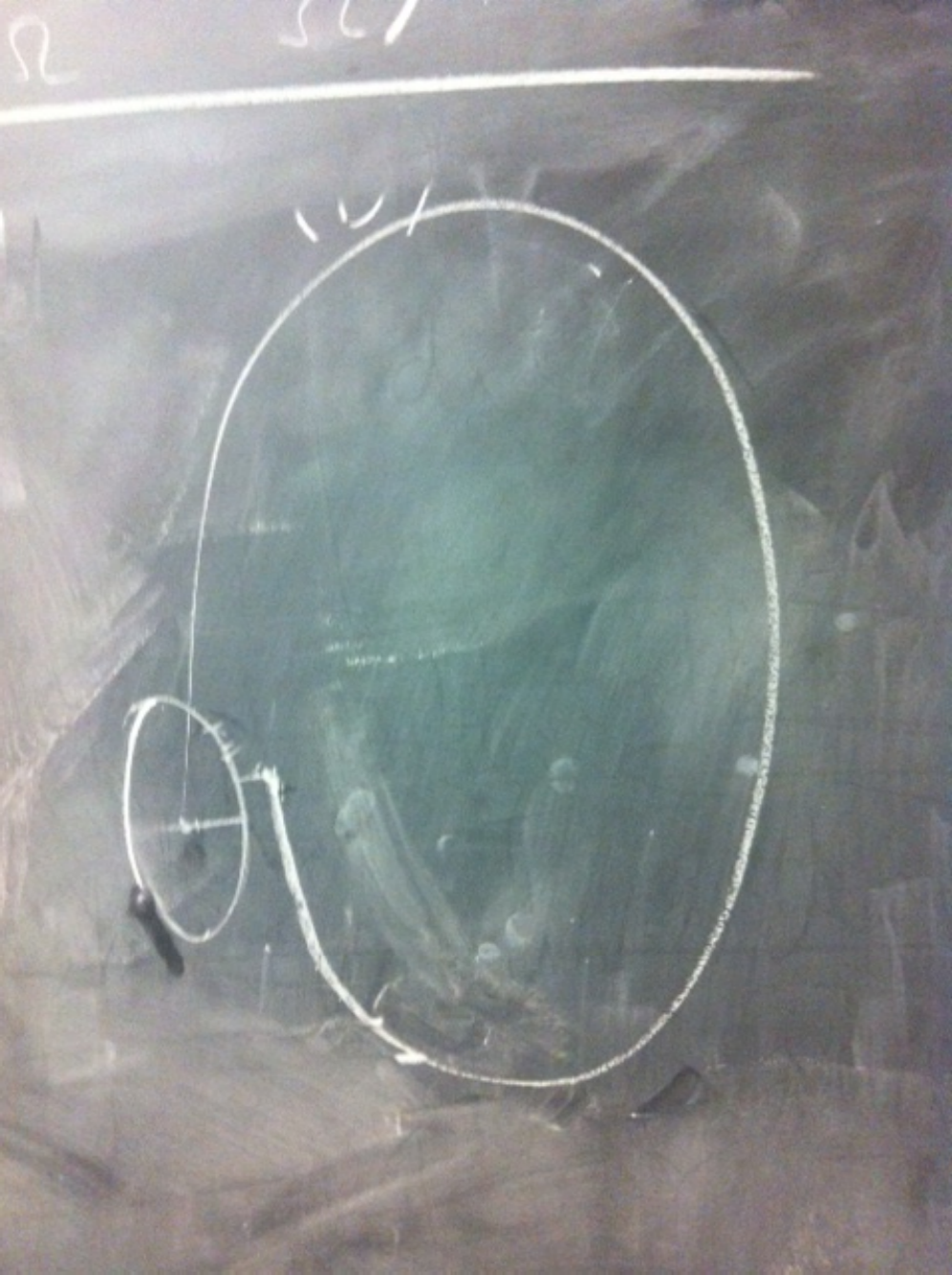}
 \end{center}\caption{Pictorial proof of Closing lemma}
\end{figure}

In this section, we consider the situation where a long geodesic comes back to a fixed $\e$-box $\pi(\fG (g_0,\e))$,
 that is, there exist $g_1, g_2\in \fG (g_0,\e)$ such that 
$$g_1 \tilde a_\gamma \tilde m_\gamma=\gamma g_2$$
for some $\gamma \in \G$ and $\tilde a_\gamma \tilde m_\gamma\in AM$ with $T:=d(\tilde a_\gamma , e)$ sufficiently large. 
The so-called closing lemma for a negatively curved space 
(see~\cite[Lemma 6.2]{M_T} and~\cite[Chapter 5]{R_T}) says that there is a closed geodesic nearby; more precisely,
 $\gamma$ is a hyperbolic element and its oriented axis $R_\gamma$
is nearby the box $\pi(\fG (g_0,\e) )$ in the space $\Gamma\ba G/M=\T^1(X)$.
We will need  more detailed information on this situation. We will show that
the oriented axis $R_\gamma$ passes through $O(\e e^{-T})$-neighborhood of the box  $\pi(\fG (g_0,\e) )$
and  $\tilde a_\gamma$ and $[\tilde m_\gamma]$  are $O(\e )$-close to $a_\gamma$ and $[m_\gamma]$ respectively
where $a_\gamma$ and $[m_\gamma]$ are
defined as in \eqref{amnot} for $\gamma$.

\begin{lem}[Effective closing lemma I] \label{eecl}
There exists $T_0\gg 1$, depending
only on $G$, for which the following holds:
For any $g_0\in G$ and any small $\e>0$, 
 suppose that there exist $g_1, g_2\in  \fG (g_0,\e)$ and $\gamma \in G$ 
 such that 
 \be\label{eq:psi-gamma'}g_1 \tilde a_\gamma
  \tilde m_\gamma =\gamma g_2\ee
  for some $\tilde a_\gamma\in A$ with $T:=d(\tilde a_\gamma, e)\ge T_0$
  and $\tilde m_\gamma\in M$.
 
Then there exists $g\in \fG(g_0, \e +O(\e e^{-T}))$ such that
$$\gamma = g a_\gamma m_\gamma g^{-1}.$$
Moreover, $a_\gamma \sim_{O(\e )} \tilde a_\gamma$, and
 $[m_\gamma] \sim_{O(\e )} [\tilde m_\gamma]$.
\end{lem}

\begin{proof} The proof is divided into two parts.

\noindent{\bf Step 1:}
We will show that for some $g_3:=g_0h_\e\in \fG(g_0,\e)$, 
\be\label{ggg} g_3^{-1} \gamma g_3 =n_w^+ a_\gamma'  m_\gamma' n_z^-\ee
where 
$a_\gamma' m_\gamma'\in
\tilde a_\gamma \tilde m_\gamma A_{O(\e)} M_{O(\e)}$,
$n_w^+\in N^+_{O(e^{-T}\e)}$ and $n_z^-\in N^-_{O(e^{-T}\e)}$.

 To prove this claim, note that there exist $b_\e, d_\e \in \fG_\e(e)$ such that
 $g_1=g_0b_\e$ and $g_2=g_0d_\e$.
Recalling the definition $$ {\fG}(\e)= (N^+_\e N^-\cap N^-_\e N^+AM) M_\e A_\e ,$$
we may write $b_\e$ and $d_\e$ as follows:
$$b_\e=b_\e^+ n_x^-  b_\e^0 \in N_\e^+ N^- (A_\e M_\e); $$
 $$d_\e =d_\e^- n_{ y}^+ d_\e^0 \in  N^-_\e N^+ (A M).$$
 By Lemma \ref{box}, we have $n_x^-\in N^-_{O(\e)}$, $n_y^+\in N^+_{O(\e)}$
 and $ d_\e^0\in A_{O(\e)}M_{O(\e)}$.
Now the equality $g_1\tilde a_\gamma \tilde m_\gamma=\gamma g_2$
can be rewritten as
\be\label{gg}  g_0 b_\e^+ n_x^- \tilde a_\gamma^{(1)} \tilde m_\gamma^{(1)}
= \gamma g_0 d_\e^- n_y^+\ee
where $a_\gamma^{(1)} \tilde m_\gamma^{(1)}:=
b_\e^0  \tilde a_\gamma \tilde m_\gamma (d_\e^0)^{-1}\in AM$.

By the transversality between $N^-$ and $AMN^+$,
we obtain a unique element $n^-_{x'}\in N^-_{O(\e)}$ satisfying that
\be \label{bh} b_\e^+ n^-_{x'}\in  d_\e^-    (N^+_{O(\e)}A_{O(\e)}M_{O(\e)}).\ee
Since $b_\e^+\in N_\e^+$ and $d_\e^-\in N_\e^-$,
we have
$$h_\e:=b_\e^+ n^-_{x'}\in (N_{\e}^+ N^- \cap N_\e^-N^+AM) \subset \fG(\e)$$
and hence $$g_0h_\e\in \fG(g_0,\e).$$

Now setting $g_3:=g_0h_\e$,
we claim that \eqref{ggg} holds. By \eqref{bh}, $h_\e= d_\e^- n_{v}^+ a_\e m_\e 
\in d_\e^-   (N^+_{O(\e)}A_{O(\e)}M_{O(\e)})$.
Rewriting \eqref{gg}, we have
$$g_0 h_\e (n_{x'}^-)^{-1} n_x^- \tilde a_\gamma^{(1)} \tilde m_\gamma^{(1)}
=\gamma g_0 h_\e m_\e^{-1}  a_\e^{-1} (n_v^+)^{-1} n_y^+   $$
and hence
\begin{align*} g_3^{-1} \gamma g_3& =
 (n_{x'}^-)^{-1} n_x^- \left( \tilde a_\gamma^{(1)} \tilde m_\gamma^{(1)}
 a_\e m_\e\right)  \left( a_\e^{-1}m_\e^{-1} (n_y^+)^{-1}n_v^+ a_\e m_\e \right) 
 \\&=n_{z_1}^- \tilde a_\gamma^{(2)} \tilde m_\gamma^{(2)}n_{w_1}^+
  \end{align*}  
where $a_\gamma^{(2)} \tilde m_\gamma^{(2)}:=  \tilde a_\gamma^{(1)} \tilde m_\gamma^{(1)}
 a_\e m_\e$, $n_{z_1}^-:= (n_{x'}^-)^{-1} n_x^-\in N^-_{O(\e)}$, and
$n_{w_1}^+:=a_\e^{-1}m_\e^{-1} (n_y^+)^{-1} n_v^+ a_\e m_\e \in N^+_{O(\e)} $.

We have $n_{z_2}^-:=  (a_\gamma^{(2)} \tilde m_\gamma^{(2)})^{-1}
n_{z_1}^- \tilde a_\gamma^{(2)} \tilde m_\gamma^{(2)}\in N^-_{O(e^{-T}\e)}$
and we can write 
$$n_{z_2}^- n_{w_1}^+= n_{w_2}^+a_\e' m_\e' n_{z_3}^-
 \in N^+_{O(\e)}A_{O(\e)} M_{O(\e)} N^-_{O(e^{-T}\e )}.$$

Therefore 
$$g_3^{-1} \gamma g_3 = \tilde a_\gamma^{(2)} \tilde m_\gamma^{(2)} n_{z_2}^-n_{w_1}^+
= n^+_{w_3} \tilde a_\gamma^{'} \tilde m_\gamma^{'} n_{z_3}^-
$$
where $a_\gamma' m_\gamma':=\tilde a_\gamma^{(2)} \tilde m_\gamma^{(2)}
 a_\e' m_\e'$, 
and $n^+_{w_3} =\tilde a_\gamma^{(2)} \tilde m_\gamma^{(2)}
n_{w_2}^+  (a_\gamma^{(2)} \tilde m_\gamma^{(2)})^{-1}\in N^+_{O(e^{-T}\e)}$.
This proves the claim \eqref{ggg}.

\medskip
\noindent{ \bf Step 2:} Set $g:=g_3^{-1} \gamma g_3$ so that
$g=n_w^+ a_\gamma'  m_\gamma' n_z^-$. We claim that
\be\label{rgg} g\in (n_x^+ n_{y}^- ) a_\gamma' A_{O(\e)} m_\gamma' M_{O(\e)}  (n_x^+ n_{y}^- )^{-1}\ee
with $n_{x}^+\in N^+_{O( \e e^{-T})}$ and $n_y^-\in N^-_{O(\e e^{-T})}$.

For $n_z^-$ as above, for any $n_x^+\in N^+_{\e}$,
there exists a unique element $n_{\alpha(x)}^+\in N^+_{O(\e)}$ such that
$(n_z^-) n_x^+ \in n^+_{\alpha(x)}A_{\e} M_\e N^-_{\e} $.
Moreover  the map $n_x^+\mapsto n_{\alpha(x)}$ 
is a diffeomorphsim of $N_\e^+$ onto its image, which is contained in $N^+_{O(\e)}$.

Therefore the implicit function theorem implies that
the map $n_x^+ \mapsto n_x^+(a_\gamma ' m_\gamma') (n_{\alpha(x)}^+)^{-1} (a_\gamma ' m_\gamma')^{-1}$
defines a diffeomorphism of $N_\e^+$ onto its image $N_{\e+O(e^{-T}\e)}^+$.
Since $n_w^+\in N^+_{O(e^{-T}\e)}$,
if $T$ is large enough, we can find $n_x^+\in N_{O(e^{-T}\e)}$ such that
$$n_w^+=n_x^+(a_\gamma ' m_\gamma') (n_{\alpha(x)}^+)^{-1} (a_\gamma ' m_\gamma')^{-1}.$$
Fixing this element $n_x^+$, we write $(n_{\alpha(x)}^+)^{-1} n_z^-=a_\e m_\e n_u^- (n_x^+)^{-1}$
with $n_u^-\in N^-_{O(e^{-T}\e)}$, $a_\e\in A_{O(\e)}$ and $m_\e\in M_{O(\e)}$.
Therefore, plugging in these,
\begin{align*}g&=n_w^+ a_\gamma'  m_\gamma' n_z^-\\
&=n_x^+ a_\gamma'  m_\gamma' (n_{\alpha(x)}^+)^{-1} n_z^-\\
&=n_x^+ (a_\gamma''  m_\gamma'') n_u^- (n_x^+)^{-1}
\end{align*}
where $a_\gamma''=a_\gamma'  a_\epsilon$ and $m_\gamma''= m_\gamma'  m_\epsilon$.
Since the map $$n_y^- \mapsto (a_\gamma'' m_\gamma'')^{-1}
  n_y^-(a_\gamma '' m_\gamma'' ) (n_y^-)^{-1} $$
is a diffeomorphism of $N_\e^-$ onto its image $N_{\e+O(e^{-T}\e)}^-$,
for all large $T$, we can find $n_y^-\in N^+_{O(\e e^{-T})}$ such that
$$n_u^-=(a_\gamma '' m_\gamma'')^{-1}
  (n_y^-)^{-1} (a_\gamma ''m_\gamma'' ) n_y^-.$$
This yields
$$g= n_x^+(n_y^-)^{-1} (a_\gamma'' m_\gamma'')  n_y^-  (n_x^+)^{-1}$$
as desired.

Hence  $$\gamma =g_4a_\gamma ''m_\gamma'' g_4^{-1}$$
with $g_4:=g_3 n_x^+ (n_y^-)^{-1}\in \fG(g_0,\e +O(\e e^{-T}))$.
Therefore 
$a_\gamma=a_\gamma '' \sim_{O(\e)} \tilde a_\gamma$
and $[m_\gamma]=[m_\gamma '']\sim_{O(\e)} [\tilde m_\gamma]$.
\end{proof}

Although we will only be using the above version of the closing lemma \ref{eecl},
we record the following reformulation as well, which is of more geometric flavor.
\begin{lem}[Effective closing lemma II]\label{ecl}
There exists $T_0\gg 1$, depending
only on $G$, for which the following holds:
Let $g_0\in G$ and let $\e>0$ be smaller than the injectivity radius of $g_0$ in $\G\ba G$.
 Suppose that there exist $g_1, g_2\in  \fG (g_0,\e)$ and $\gamma\in \G$ 
 such that 
 \be\label{eq:psi-gamma}g_1 \tilde a_\gamma
  \tilde m_\gamma =\gamma g_2\ee
  for some $\tilde a_\gamma\in A$ with $T:=d(\tilde a_\gamma , e)\ge T_0$
  and $\tilde m_\gamma\in M$.
 Suppose also that $\gamma$ is primitive, i.e., $\gamma$ cannot be written as a power
 of another element of $\Gamma$.
Then there exists an element $g_\gamma\in \fG(g_0, \e +O(\e e^{-T}))$ such that
 \begin {enumerate}
\item the $AM$-orbit $\G\ba \G  g_\gamma AM$  is compact;
\item $\gamma$ is a generator of the group $g_\gamma AM g_\gamma^{-1}\cap \G$;
\item the length of the closed geodesic $C_\gamma=\G\ba \G  g_\gamma A(v_o)$
is $T+O(\e)$;
\item the holonomy class $[m_\gamma]$ associated to $C_\gamma$ is
 within $O(\e)$-distance from $ [\tilde m_\gamma]$.
\end{enumerate}
\end{lem}


\section{Counting results for  $\G\cap \fG (g_0,\e)A_T \Omega \fG(g_0,\e)^{-1}$}\label{sec:bisec-count}\label{sc}
Let $G, \Gamma, X$, $o, v_o$ etc be as in the previous section. Recall $A_T^+=\{a_t: 0<t<T\}$. 
Our approach of understanding the distribution of closed geodesics in $\T^1(X)$
passing through the flow box $\fG(g_0, \e)$ and with holonomy class
contained in a fixed compact subset $\Omega$ of $M$ is to interpret it as a counting problem
for the set $\G\cap \fG (g_0, \e)A_T^+ \Omega \fG(g_0, \e)^{-1}$ as $T\to \infty$.
We will be able to approximate $\# \G\cap \fG (g_0, \e)A_T^+ \Omega \fG(g_0, \e)^{-1}$
 by the counting function for the intersection of $\G$ with a certain compact subset given in the $g_0N^+AMN^-g_0^{-1}$ coordinates.

 In the first part of this section,  we will investigate the asymptotic behavior of the following
$$\# \Gamma\cap g_0\Xi_1 A_T^+\Omega \Xi_2 g_0^{-1} $$
for given bounded Borel subsets $\Xi_1\subset N^+$, $\Xi_2\subset  N^-$ and $\Omega\subset M$.
In the second part, we will use this result to obtain an asymptotic formula of $\# \G \cap
\fG (g_0,\e)A_T^+ \Omega \fG(g_0,\e)^{-1}$.

\subsection{On the counting for $\Gamma\cap g_0\Xi_1 A_T\Omega \Xi_2 g_0^{-1} $}This problem can be answered under the extra assumption that $\G$ is Zariski dense
and that the Bowen-Margulis-Sullivan measure, the BMS measure for short,
 on $\T^1(X)=\Gamma\ba G/M$ is {\it finite}.
The key ingredient in this case is that the $M$-invariant extension of the BMS measure on
$\G\ba G$ is mixing for the $A$-action.

We begin the discussion by recalling the definition of the BMS measure.
Let $\Lambda(\G)$ denote the limit set of $\Gamma$, which  is the set of all accumulation points
in $\tilde X\cup \partial(\tilde X)$ of an orbit of $\Gamma$ in $\tilde X$.
 Denote by $\delta=\delta_\G$ the critical exponent of $\Gamma$.
Denote by $\{\nu_x: x\in \tilde X\}$ a $\G$-invariant conformal density of dimension $\delta$ 
 supported on the limit set $\Lambda(\G)$; such a density exists by the construction given by Patterson \cite{Pa}.
For $\xi_1\ne \xi_2\in \partial(\tilde X)$, and $x\in \tilde X$,
we denote by  $\la \xi_1, \xi_2\ra_x$ the Gromov product at $x$. Then
the visual distance on $\partial(\tilde X)$ at $x$ is given by
 $$d_x(\xi_1, \xi_2)=e^{-\la \xi_1 , \xi_2\ra_x}$$
with the convention that $d_x(\xi, \xi)=0$. The Hopf parametrization of $\T^1(\tilde X)$
 as $(\partial^2 (\tilde X) \setminus \text{Diagonal})\times \br$ is given by
 $u\mapsto (u^+, u^-, s=\beta_{u^-}(o, u))$ where $\beta_{\xi}(x,y)$ denotes the Busemann function for $\xi\in \partial(\tilde X)$,
and $x,y\in \tilde X$.
The BMS measure
on $\T^1(\tilde X)$ with respect to $\{\nu_x\}$ is defined as follows:
 $$d \tilde m^{\BMS}(u)= \frac{d\nu_x(u^+) d\nu_x(u^-) ds}{d_x(u^+, u^-)^{2\delta}}.$$
 
 The definition is independent of $x\in \tilde X$ and $\tilde m^{\BMS}$ is right $A$-invariant and left $\G$-invariant, and hence
 induces a geodesic flow invariant Borel measure on $\T^1(X)$, which we denote by $m^{\BMS}$.
If $|m^{\BMS}|<\infty$, then  the geodesic flow is ergodic with respect to $m^{\BMS}$, as shown by Sullivan
\cite{Sullivan1979} and moreover mixing by Babillot \cite{Bab}.

As we are eventually interested in counting a $\G$ orbit in a family $\Xi_1A_T^+ \Omega \Xi_2$ 
with $\Omega$ any Borel subset in $M$, we need to understand the mixing phenomenon for
the $A$-action on $\Gamma\ba G$, not only on $\Gamma\ba G/M$.
By abuse of notation, we denote by $m^{\BMS}$  the $M$-invariant lift of $m^{\BMS}$ to $\Gamma\ba G$.
Winter \cite{Wi} showed that if $\G$ is Zariski dense and $|m^{\BMS}|<\infty$, then
the $A$-action on $\Gamma\ba G$ is mixing for this extension $m^{\BMS}$; this was earlier claimed in \cite{FS} for
the case of $G=\SO(n,1)^\circ$ and $\G$ geometrically finite.

In the rest of this section, we assume that
$$\text{$\G$ is Zariski dense and $|m^{\BMS}|<\infty$}.$$
For the application of the mixing in counting problems, it is easier to use
the following version on the asymptotic behavior of the matrix coefficients in Haar measure.
 To state this result, we need to recall the Burger-Roblin measures for the $N^+$ and $N^-$ actions.


 Using the homemorphism of
 $G$ with $K/M\times M\times A\times N^{\pm}$,
we define the Burger-Roblin measures $\tilde m^{\BR}$ (invariant under the $N^+$-action) and
$\tilde m_*^{\BR}$ (invariant under the $N^-$-action) on $ G$ as follows: 
\be\label{brs} d\tilde m^{\BR}(kma_r n^+)=e^{-\delta r} dn^+drd\nu_o(kv_o^-) dm \;\;\text{ for $k m a_rn^+\in (K/M) MAN^+$};\ee
\be \label{brf} d\tilde m_*^{\BR}(kma_r n^-)=e^{\delta r} dn^-drd\nu_o(k v_o^+)dm \;\; \text{for $km a_rn^-\in (K/M) MAN^-$}\ee
where $dm$ denotes the $M$-invariant probability measure on $M$;
Since $M$ fixes $v_o$ and hence  fixes $v_o^{\pm}$, these measures are well-defined.
The Haar measure on $G$ is given by: for $g=a_sn^{\pm} k\in AN^{\pm}K$,
$$dg=d\tilde m^{\Haar}(a_sn^{\pm} k)= dsdn^{\pm}dk$$ where $dk$ is the probability Haar measure
on $K$.
 These measures are all left $\G$-invariant and we use the notations $m^{\BR}$, $m_*^{\BR}$, $m^{\Haar}$ (or $dg$) respectively  for
the corresponding induced right $M$-invariant measures on $\G\ba G$.

The following theorem can be deduced from the mixing of $m^{\BMS}$, as observed first by Roblin for $M$-invariant functions (\cite{R_T},
see also \cite{OS}). 
\begin{thm} $($\cite{R_T}, \cite{OS}, \cite{Wi}$)$ \label{lo} \label{harmixing} 
For any functions $\Psi_1,\Psi_2\in C_c(\G\ba G)$,
$$\lim_{t\to +\infty}
e^{(D -\delta)t} \int_{\G\ba G}
 \Psi_1(ga_t) \Psi_2(g) dg  = \frac{m^{\BR}(\Psi_1)\cdot m^{\BR}_*(\Psi_2)}{|m^{\BMS}|} $$
where $D=D(\tilde X)$ is the volume entropy of $\tilde X=G/K$ (see \eqref{ve1} and \eqref{ve2}).
\end{thm}

The quotient by $\Gamma$ of the convex hull of $\Lambda(\G)$  is called the convex core  of $\Gamma$.
A discrete group $\G$ is called {\it geometrically finite} if the volume of a unit neighborhood
of the convex core of $\Gamma$ is finite. Clearly lattices are geometrically finite.
 If $\Gamma$ is geometrically finite,
then $m^{\BMS}$ is known to be finite  and the critical exponent is known to be equal
to the Haudorff dimension of $\Lambda(\G)$ (\cite{Sullivan1984} and \cite{CI}).

We use the standard asymptotic "big-O" and "little-o" notations, where
for functions $f, g : \br^+ \to \br^+$, we write $f=O(g)$ if $\limsup_{T} f(T)/g(T)<\infty$
and $f=o(g)$ if $\lim_Tf(T)/g(T)=0$. We sometimes write $f=O_T(g)$ and $f=o_T(g)$
in order to clarify the parameter $T$ going to infinity. The notation $f(T)\sim g(T)$ means that
$\lim_{T\to \infty} f(T)/g(T)=1$.


\begin{thm}\cite{MO}  \label{hem}  Suppose that $\G$ is a
geometrically finite subgroup of $\SO(n,1)^\circ$ with $n\ge 2$. Suppose that
$\delta>(n-1)/2$ if $n=2,3$ and that $\delta>n-2$ if $n\ge 4$.
Then there exists $\e_0>0$ such that
for any functions $\Psi_1,\Psi_2\in C^\infty_c(\G\ba G)$, as $t\to +\infty$,
$$
e^{(n-1 -\delta)t} \int_{\G\ba G}
 \Psi_1(ga_t) \Psi_2(g) dg  = \frac{m^{\BR}(\Psi_1)\cdot m^{\BR}_*(\Psi_2)}{|m^{\BMS}|} +O(e^{-\e_0 t})$$
where the implied constant depends only on the Sobolev norms of $\Psi_1$ and $\Psi_2$.
\end{thm}

Let $\Omega\subset M$, $\Xi_1\subset N^+$ and $\Xi_2\subset N^-$ 
be bounded Borel subsets.
For $T>0$, set
\be\label{eq:ccal-T}
\cS_T(\Xi_1,\Xi_2, \Omega)=\Xi_1A_T^+\Omega \Xi_2 .
\ee 

For Borel subsets $\Xi_1\subset N^+$ and $\Xi_2\subset N^-$ and all $g_0\in G$, put
\begin{align*}
&\hat\nu_{g_0}^+(\Xi_1v_o^+)=\int_{\Xi_1}e^{\delta r(n_1)}d\nu_{g_0(o)}(g_0n_1v_o^+)\quad\text{ and }\\ 
&\hat\nu_{g_0}^-(\Xi_2v_o^-)=\int_{\Xi_2}e^{-\delta r(n_2)}d\nu_{g_0(o)}(g_0n_2v_o^-)
\end{align*}
where $n_1=k_1a_{r(n_1)}m_1n^-\in KP^-$ and $n_2=k_2a_{r(n_2)}m_2n^+\in KP^+$.

By $\text{Vol}(\Omega)$, we mean the volume of $\Omega$ computed with respect to the
probability Haar measure on $M$.

\begin{thm}\label{count} 
Fix $g_0\in G$.  If $\nu_{o}(\partial (\Xi_1 v_o^+))=0= \nu_{o}(\partial(\Xi_2^{-1} v_o^-)) $
and $\op{Vol}(\partial(\Omega))=0$, then  as $T\to \infty$,
$$\# \G \cap g_0 \cS_T(\Xi_1,\Xi_2, \Omega) g_0^{-1} \sim \frac{ \hat\nu_{g_0}(\Xi_1v_o^+) \hat\nu_{g_0}(\Xi_2^{-1}v_o^-) \op{Vol}(\Omega) }
{\delta |m^{\BMS}|} e^{\delta T}.$$
\end{thm}

Under the assumption of Theorem \ref{hem}, we will prove an effective version of Theorem \ref{count}.
As usual, in order to state a result which is effective, we need to assume certain regularity condition on the boundaries
of the sets $\Xi_1,\Xi_2, \Omega$ involved.

\begin{dfn}\rm
A Borel subset $\Theta \subset \partial(\tilde X)$ is called {\it admissible} with respect to $\nu_o$ if there exists $r>0$ such that
for all small $\rho>0$,
 $$\nu_o\{\xi\in \partial(\tilde X) : d_o(\xi, \partial(\Theta))\le \rho \} \ll \rho^r$$
\end{dfn}

\begin{rmk}\label{rem1}\rm In the group $G=\SO(2,1)^\circ$, the boundary $\partial(\tilde X)$ is a circle, and
any interval of  $\partial(\tilde X)$ is admissible.
For $G=\SO(n,1)^\circ$ with $n\ge 3$,
if $\delta> \max\{n-2, (n-2+\kappa)/2\} $ where $\kappa$
is the maximum rank of parabolic fixed points of
$\G$, then any Borel subset $\omega$ of $\partial(\tilde X)$ such that
$\nu_o(\omega)>0$ and $\partial(\omega)$ is a finite union of smooth sub manifolds is admissible; this is proved
in \cite{MO}, using Sullivan's shadow lemma.\end{rmk}


\begin{thm} \label{ecount} Let $G$ and $\G$ be as in Theorem \ref{hem}.
Suppose that $\Xi_1 v_o^+$ and $\Xi_2^{-1} v_o^{-}$ are admissible, and
that $\partial(\Omega)$ is a finite union of smooth submanifolds. 
Then for any $g_0\in G$, there exists $\e_0>0$ such that as $T\to \infty$,
\[\# \G \cap g_0 \cS_T(\Xi_1,\Xi_2, \Omega) g_0^{-1}=
 \frac{ \hat\nu_{g_0}(\Xi_1v_o^+) \hat\nu_{g_0}(\Xi_2^{-1}v_o^-) \op{Vol}(\Omega) }{\delta |m^{\BMS}|}  e^{\delta T} 
 +O(e^{(\delta -\e_0)T}).
 \]
\end{thm}

The rest of this section is devoted to the proof of Theorems \ref{count} and \ref{ecount}.
In the case when $G=\SO(n,1)^\circ$, an analogous theorem for bisectors in $KA^+K$  was proved
in \cite{MO}  (see also \cite{BKS} for $n=2$, \cite{V} for $n=3$ when $\delta$ is big and \cite{GO} when $\Gamma$ is a lattice).
 In view of Theorem \ref{lo}   for a general rank one homogeneous space
admitting a finite BMS measure, the proof of Theorem \ref{count} is very similar to the one given in \cite{MO} in principle. 

For simplicity, we normalize $|m^{\BMS}|=1$ by replacing $\nu_o$ by a suitable scalar multiple.
For a given compact subset $\B\subset G$, consider the following function on $\G\ba G\times \G\ba G$: 
$$F_{\B}(g,h):=\sum_{\gamma\in \G} 1_{\B}(g^{-1}\gamma h) .$$

Note that for  $\Psi_1,\Psi_2 \in C_c(\G\ba G)$
\begin{align*} \la F_{\B}, \Psi_1\otimes \Psi_2\ra_{\G\ba G\times \G\ba G}
&:=\int_{\G\ba G\times \G \ba G}  F_{\B}(g_1, g_2)
\Psi_1(g_1)\Psi_2(g_2) dg_1 dg_2 .\end{align*}
By a standard folding and unfolding argument, we have
$$\la F_{\B}, \Psi_1\otimes \Psi_2\ra
 =\int_{g\in {B}} \la \Psi_1, g.\Psi_2\ra_{L^2(\G\ba G)} \;  dg .$$
Let $\psi^\e\in C^\infty(G)$ be an $\e$-approximation function of $e$, i.e.,
 $\psi^\e$ is a non-negative smooth function supported on $G_\e(e)$ and $\int \psi^\e dg=1$,
   and let $\Psi^\e\in C^\infty(\G\ba G)$ be
its $\G$-average: $\Psi^{\e}(\Gamma g)=\sum_{\gamma\in \G}\psi^{\e}(\gamma g)$.

We deduce that 
\begin{align}\label{forr}
 &\la F_{\B}, \Psi^\e\otimes \Psi^\e_{{}}  \ra_{\G\ba G\times \G\ba G}
 \notag \\ &=\int_{x\in \B}\int_{\G\ba G} \Psi^\e(g) \Psi^\e_{{}}(gx)dg dx\notag \\
&\text{ writing $x=n_1a_tmn_2\in N^+AMN^-$ and using $dx=e^{Dt} dn_1 dt dmdn_2$}\notag \\
 &=\int_{n_1 a_t m n_2\in \B}\left(\int_{\G\ba G} \Psi^\e(g) \Psi^\e_{{}}(gn_1 a_tmn_2) dg \right) e^{D t} 
 dtdn_1 dm dn_2\notag 
 \\&=\int_{n_1 a_t m n_2\in \B}\left(\int_{\G\ba G} \Psi^\e(gn_1^{-1}) \Psi^\e_{{}}(ga_tmn_2) dg \right) e^{D t} 
 dtdn_1 dm dn_2\notag \\ &\text{by applying Theorem \ref{harmixing} }\notag \\
 &= \int_{n_1 a_t m n_2\in \B} e^{\delta t} (1+o(1)) 
 m^{\BR}_{*}(n_1^{-1}\Psi^\e) m^{\BR}((mn_2) \Psi_{{}}^\e) dt dn_1dm dn_2 
\notag \\ &=\int_{n_1a_tmn_2\in \B} e^{\delta t} (1+o(1))  \tilde m^{\BR}_{*}(n_1^{-1}\psi^\e) \tilde m^{\BR}(mn_2 \psi^\e) dt 
dn_1 dm dn_2,
\end{align}
provided that $n_1$ and $n_2$ are from a fixed bounded subset of $N^+$ and $N^-$, respectively.  

If we define a function $f_{\B}$ on $N^+ \times MN^-$  by
 $$f_{\B}(n_1, m n_2)=\int_{a_t\in n_1^{-1}\B n_2^{-1}m^{-1} \cap A^+} e^{\delta t} dt ,$$
and a function on $G\times G$ by
\begin{multline*}
\left((\psi^\e\otimes \psi^\e)*f_\B \right)(g,h)\\
=\int_{n_1mn_2\in N^+MN^-} \psi^\e(gn_1^{-1}) \psi^\e(hmn_2) f_\B(n_1, mn_2) dm dn_1 dn_2\\
=\int_{n_1mn_2\in N^+MN^-} \psi^\e(gn_1) \psi^\e(hmn_2) f_\B(n_1^{-1}, mn_2) dm dn_1 dn_2,\end{multline*}
then we may write
\begin{align}\label{when}
&\la F_{B}, \Psi^\e\otimes \Psi^\e_{{}}  \ra_{\G\ba G\times \G\ba G}
\\
&=  ( \tilde m_*^{\BR}\otimes \tilde m^{\BR}) ((\psi^\e\otimes \psi^\e)*f_{B})
+ o(\max_{n_1a_t mn_2\in B} e^{\delta t})  \notag.
 \end{align}

Observe that
\begin{multline}\label{mbrf} ( \tilde m_*^{\BR}\otimes \tilde m^{\BR}) ((\psi^\e\otimes \psi^\e)*f_{\B})=\\
\int_{N^+MN^-} f_{\B}(n_1^{-1}, mn_2) \left(\int_{G\times G}\psi^\e (g_1n_1) \psi^\e(h_1mn_2)  d\tilde m_*^{\BR}(g_1)d\tilde m^{\BR}(h_1)\right)  
 dn_1 dm dn_2 .\end{multline}

By \eqref{brs} and \eqref{brf}, we have
 \be\label{inser} d\tilde m_*^{\BR}(g_1)d\tilde m^{\BR}(h_1)=
e^{\delta (r-r_0)} dn dr dm_1 d\nu_o(kv_o^+) dn_0 dr_0 dm_0 d\nu_o(k_0 v_o^-).\ee
for $g_1= km_1a_rn\in (K/M) MAN^- $ and $h_1= k_0m_0a_{r_0}n_0\in (K/M) MAN^+$.

For $x\in G$, let $\mathfrak n_1(x)$ be the $N^+$ component of $x$ in $MAN^-N^+$ decomposition and
$\tilde{\mathfrak n_2}(x)$ be the $MN^-$ component of $x$ in $AN^+(MN^-)$ decomposition.
The $A$-components of $x$
in $MAN^-N^+$ and $AN^+ MN^-$ decompositions are respectively denoted by $I_1(x)$ and $I_2(x)$.

Continuing \eqref{mbrf}, first change the inner integral using \eqref{inser} and then
perform the change of variables by putting
 $g=m_1a_r n n_1\in MAN^-N^+$ and $h= a_{r_0} n_0 m n_2\in AN^+MN^-$.
 Since
$dg=dm_1 dr dndn_1$ and $dh=dr_0 dn_0dm dn_2$, we obtain

\begin{align}\label{fbbb}
&( \tilde m_*^{\BR}\otimes \tilde m^{\BR}) ((\psi^\e\otimes \psi^\e)*f_{\B})=
\\&
\int_{k\in K/M, k_0\in K/M, m_0\in M} \int_{G\times G}\psi^\e (kg) \psi^\e(k_0m_0h)f_{\B}(\mathfrak n_1(g)^{-1}, \tilde{\mathfrak n_2}(h)) e^{\delta (I_1(g)-I_2(h))}
\notag\\&  dgdh  d\nu_o(kv_o^+) d\nu_o(k_0v_o^-) dm_0
\notag\\
 &=
\int_{k\in K/M, k_0\in  K} \int_{G\times G}\psi^\e (g) \psi^\e(h)f_{\B}(\mathfrak n_1 (k^{-1}g)^{-1},\tilde{\mathfrak n}_2(k_0^{-1} h))\notag \\&  e^{\delta (I_1(k^{-1}g)-I_2(k_0^{-1}h))}
 dgdh  d\nu_o(kv_o^+) d\nu_o(k_0)  . \notag
\end{align} where  $d\nu_o(k_0):=d\nu_o(k_0' v_o^-)dm$ for $k_0=k_0'\times m \in K/M\times M$.

\medskip
In order to prove Theorem \ref{count},
we now put $$\cS_T:=\cS_T(\Xi_1, \Xi_2, \Omega)\text{ and } F_T:=F_{\cS_T} .$$
Observe that $$F_T(e,{e} )=\# (\G {}\cap \cS_T(\Xi_1, \Xi_2, \Omega)) .$$
Let
$$\cS_{T,\e}^+=\cup_{g_1, g_2\in G_\e(e)} g_1 {\cS_T} g_2\;\text{ and }\;\cS_{T,\e}^-=\cap_{g_1, g_2\in G_\e(e)} g_1{\cS_T} g_2.$$
 We then have
\be\label{compa} 
\la F_{\cS_{T,\e}^-}, \Psi^\e\otimes \Psi^\e_{{}}  \ra \le
F_T(e, {e})\le  \la F_{\cS_{T,\e}^+}, \Psi^\e\otimes \Psi^\e_{{}}  \ra .
\ee

Using our assumptions $\nu_{o}(\partial (\Xi_1 v_o^+))=\nu_{o}(\partial(\Xi_2^{-1} v_o^-))=\op{Vol}(\partial(\Omega))=0$, 
and the strong wave front property for the $AN^{\pm}K$ decompositions \cite{GOS},  
\eqref{fbbb} applied with $B=\cS_{T,\e}^{\pm}$ now implies that
\begin{align}\label{ep} 
&( \tilde m_*^{\BR}\otimes \tilde m^{\BR}) ((\psi^\e\otimes \psi^\e)*f_{\cS_{T,\e}^\pm})\\ 
\notag&=(1+O(\e')) \int_{} e^{\delta (I_1(k^{-1})-I_2(k_0^{-1}))}f_{{\cS_T}}( \mathfrak n_1 (k^{-1})^{-1},\tilde{\mathfrak n}_2(m^{-1}k_0^{-1} ) )  d\nu_o(kv_o^+)  d\nu_o(k_0)
\end{align}
where we integrate over $K/M\times K$ and $\e'>0$ goes to $0$ as $\e\to 0$.

Recall now that if we write $k^{-1}=ma_{r}n^-n_1$, then $\mathfrak n_1 (k^{-1})=n_1$ and $I_1(k^{-1})=r$.
This also implies that $n_1^{-1}=kma_rn^{-}$. Therefore, $r(n_1^{-1})=I_1(k^{-1})$, which implies that
\[
\hat\nu_o^+(\Xi_1v_o^+)=\int_{\Xi_1} e^{\delta r(n)}d\nu_o(nv_0^+)=\int_{\mathfrak n_1 (k^{-1})^{-1}\in \Xi_1} e^{\delta I_1(k^{-1})}d\nu_o(kv_0^+)
\]
Similarly, if we write $m^{-1}k_0^{-1}=a_{r'}n^+ m_2n_2$, then $\mathfrak n_2 (m^{-1}k_0^{-1})=m_2n_2$ and $I_2(k_0^{-1})=r'$.
This also implies that $n_2^{-1}=k_0ma_{r'}n^+m_2$. Therefore, $r(n_2^{-1})=I_2(k_0^{-1})$ and we have 
\[
\hat\nu_o^-(\Xi_2^{-1}v_o^-)=\int_{\Xi_2^{-1}} e^{-\delta r(n)}d\nu_o(nv_0^-)=\int_{\mathfrak n_2 (m^{-1}k_0^{-1})^{-1}\in \Xi_2^{-1}M} 
e^{-\delta I_2(k_0^{-1})}d\nu_o(k_0v_0^-)
\] 
These together with~\eqref{ep} imply that 

\[
( \tilde m_*^{\BR}\otimes \tilde m^{\BR}) ((\psi^\e\otimes \psi^\e)*f_{\cS_{T,\e}^\pm})=(1+O(\e')) \frac{ e^{\delta T}}{\delta}  
\hat\nu_o^+ (\Xi_1 v_o^+) \hat\nu_o^- (\Xi_2^{-1} v_o^-) \op{Vol}(\Omega)
\] 
where $\e'>0$ goes to $0$ as $\e\to 0$.
Hence, \eqref{compa}, \eqref{forr} and \eqref{when} yield that
\[
F_T(e,{e})= (1+O(\e')) \frac{ e^{\delta T}}{\delta} \hat\nu_o^+ (\Xi_1 v_o^+) \hat\nu_o^- (\Xi_2^{-1} v_o^-)\op{Vol}(\Omega) +o(e^{\delta T}).
\]
Since $\e>0$ is arbitrary,
we have
$$F_T(e,e)\sim  \frac{ e^{\delta T}}{\delta}  \hat\nu_o^+ (\Xi_1 v_o^+) \hat\nu_o^- (\Xi_2^{-1} v_o^-) \op{Vol}(\Omega) .$$
In order to prove Theorem \ref{ecount} for $g_0=e$, we note that $o(e^{\delta T})$
in \eqref{forr} can be upgraded into $O(e^{(\delta -\e_0) T})$ in view of Theorem \ref{hem}, 
 and that $O(\e')$  in \eqref{ep} can be taken as $O(\e^{q})$
for some fixed $q>0$ (we refer to \cite{MO} for details).

Therefore we get
$$F_T(e,{e})= (1+O(\e^q)) \frac{ e^{\delta T}}{\delta}  \hat\nu_o^+ (\Xi_1 v_o^+) \hat\nu_o^- (\Xi_2^{-1} v_o^-)  \op{Vol}(\Omega) +
O(e^{(\delta-\e_0)T}) .$$

By taking $\e$ so that $\e^q=e^{-\e_0 t}$, we then obtain 
\[
F_T(e,{e})= \frac{ e^{\delta T}}{\delta}  \hat\nu_o^+ (\Xi_1 v_o^+) \hat\nu_o^- (\Xi_2^{-1} v_o^-)  \op{Vol}(\Omega) +
O(e^{(\delta-\e_1)T}) 
\]
for some positive $\e_1>0$. This proves Theorems \ref{count} and \ref{ecount} for $g_0=e$.

For a general $g_0\in G$, we note that if we set $\G_0:=g_0^{-1}\G g_0$, then 
$$\# \G \cap g_0 \cS_T(\Xi_1,\Xi_2, \Omega) g_0^{-1}=
\# \G_0\cap  \cS_T(\Xi_1,\Xi_2, \Omega) .$$
Moreover,
if we set $\nu_{\G_0, x}:=g_0^* \nu_{ g_0(x)}$, then $\{\nu_{\G_0, x}:x\in \tilde X\}$
is a $\G_0$-invariant conformal density of dimension $\delta =\delta_{\G_0}$, and the corresponding BMS-measure  $m^{\BMS}_{\G_0}$
with respect to $\{\nu_{\G_0, x}\}$
has the same total mass as $m^{\BMS}$.
Therefore 
$$\frac{ \hat\nu_{\G_0, o}^+(\Xi_1 v_o^+) \hat\nu_{\G_0, o}^-(\Xi_2^{-1} v_o^-)  }
{\delta_{\G_0} |m^{\BMS}_{\G_0}|}=  \frac{ \hat\nu_{g_0(o)}^+(g_0\Xi_1 v_o^+) \hat\nu_{g_0(o)}^-(g_0\Xi_2^{-1} v_o^-)  }
{\delta |m^{\BMS}|}.$$

Hence the general case follows from $g_0=e$.

\subsection{On the counting for $\G\cap \fG (\e)A_T^+ \Omega \fG(\e)^{-1}$}
Recall the definition of our flow box at $g_0\in G$ with $\e>0$ smaller than the 
injectivity radius of $g_0$ in $\G\ba G$:
\be\label{fgdd} \mathfrak{B}(g_0, \e)=g_0 (N^+_\e N^-\cap N^-_\e N^+AM) M_\e A_\e .\ee

Denote $\tilde \pi: G\to \G\ba G/M$ the canonical
projection map. For simplicity, we set 
\be\label{tfb} \tilde \fG(g_0, \e)=\tilde \pi (\fG(g_0, \e)).\ee

For a Borel function $f$ on $\G\ba G/M$ and a Borel function $\xi$ on $M$,
we set $$m^{\BMS}(f\otimes \xi): =\int_{\T^1(X)} f dm^{\BMS}\cdot  \int_M \xi dm; $$
For Borel subsets $B\subset \G\ba G/M$ and $\Omega\subset M$,
we set $m^{\BMS}
( B\otimes \Omega)=m^{\BMS}(1_B\otimes 1_\Omega)$.
We observe that:
\begin{lem}\label{bms1} For all small $\e>0$,
$$ m^{\BMS}
( \tilde \fG(g_0, \e)\otimes \Omega)=(1+O(\e)) 2\e \cdot \hat\nu_{g_0(o)}^+(g_0 N^+_\e v_o^+) 
\hat\nu_{g_0(o)}^-(g_0 N^-_\e v_o^-)\op{Vol}(\Omega)$$
where the implied constant is independent of $\e>0$. \end{lem}
\begin{proof} 
Clearly we have $m^{\BMS}( \tilde \fG(g_0, \e)\otimes \Omega) =\tilde m^{\BMS}( \fG(g_0, \e)\otimes \Omega)$.
Recall that the BMS measure on $\T^1(\tilde X)$ is given as
$$d\tilde m^{\BMS}(u)=\frac{d\nu_{g_0(o)} (u^+) d\nu_{g_0(o)} (u^-) ds}{d_{g_0(o)} (u^+, u^-)^{2\delta}}.$$
Note that $$\fG(g_0, \e)v_o^+ = g_0N_\e^+v_o^+$$ (which is equal to the image of $\fG(g_0, \e)$ in $G/(MAN^-)$) and
$$\fG(g_0, \e)v_o^- = g_0N_\e^-v_o^-$$ (which is equal to the image of $\fG(g_0, \e)$ in $G/(MAN^+)$) .
Hence for  all $g\in \fG(g_0, \e)$, we have
$d_{g_0(o)} (g^+, g^-)=(1+O(\e))$
where the implied constant is independent of $g_0\in G$ and $\e>0$.
Moreover, for all $g\in \fG(g_0, \e)$, $\{t\in \br: ga_t\in  \fG(g_0, \e)\}$ has
length precisely $2\e$ (see Lemma \ref{box}).
Therefore the claim follows, since the BMS measure on $G$ is the $M$-invariant extension of the BMS measure of
$G/M$.
\end{proof}

For $T>1$ and $g_0\in G$, we define 
\be \label{vgd} \mathcal V_T(g_0,\e, \Omega):=
\fG(g_0, \e  ) A_T^+ \Omega  \fG(g_0, \e )^{-1}.\ee
We set
$$\mathcal V_T( \e, \Omega):=\V_T(e,\e, \Omega)$$
and note that $$ \mathcal V_T(g_0, \e,\Omega):=
g_0\mathcal V_T(\e,\Omega)g_0^{-1}.$$

\begin{lem}\label{vcc} For all large $T\gg 1$ and small $0<\e <1$, we have
$$\cS_T( N^+_{\e}, (N^-_{\e})^{-1}, \Omega) \subset \vcal_T (\e,\Omega)\subset \cS_{T+\e}( N^+_{\e +e^{-T}}, (N^-_{\e -e^{-T}})^{-1}, \Omega^+_\e)$$
where $\Omega_\e^+=\cup_{m_i\in M_{\e}} m_1 \Omega m_2$. \end{lem}
\begin{proof} Given $g\in \mathfrak B(\e)\cup \mathfrak B(\e)^{-1}$,
we decompose 
$$g=g_+g_0g_-\in N^+ (AM )N^-.$$
It easily follows from the definition of $\mathfrak B(\e)$ that
$$N_\e^+=\{g_+: g\in \mathfrak B(\e)\}
\quad \text{ and }\quad (N_\e^-)^{-1}=\{g_-: g\in \mathfrak B(\e)^{-1}\}.$$

Hence 
\be\label{cl} \cS_T( N^+_{\e}, (N^-_{\e})^{-1}, \Omega) \subset \vcal_T (\e,\Omega).\ee
On the other hand,
if $g_1\in \mathfrak B(\e)$, $g_2\in \mathfrak B(\e)^{-1}$,  $a\in A_T^+$, and $m\in M$,  then
$$g_1amg_2 \in (g_1)_+ N_{e^{-T}}^+ am A_\e M_\e (N^-_{e^{-T}})^{-1} (g_2)_-.$$
Therefore 
\be\label{cu}
 \vcal_T (\e,\Omega)
\subset \cS_{T+\e}( N^+_{\e +e^{-T}},( N^-_{\e -e^{-T}})^{-1}, \Omega^+_{\e}).\ee
This proves the claim.
\end{proof}

\begin{thm}\label{bef3} Let $\e>0$ be smaller than the injectivity radius of $g_0$.
We have 
$$\# \G \cap \mathcal V_T( g_0, \e, \Omega) 
= (1+O(\e))\frac{ e^{\delta T}}{\delta\cdot 2\e\cdot |m^{\BMS}|} \cdot (  m^{\BMS}
(  \fG(g_0,\e)\otimes \Omega) +o(1)) $$
where the implied constants are independent of $\e$. 

Moreover if $G$ and $\G$ are as in Theorem \ref{m2},
 $o(1)$ can be replaced by $O(e^{-\e_1 T})$ for some positive $\e_1>0$.
\end{thm}

\begin{proof} 
By  Lemma \ref{vcc}, we have
 $$g_0 \cS_T(N^+_{\e}, (N^-_{\e})^{-1}, \Omega) g_0^{-1}\subset \vcal_T (g_0,
 \e,\Omega)\subset g_0\cS_{T}( N^+_{\e +e^{-T}}, (N^-_{\e -e^{-T}})^{-1}, \Omega^+_\e)g_0^{-1}.$$
 By Theorem \ref{count} and Lemma \ref{bms1},
 we have
 \begin{multline*}
\# \G \cap \mathcal V_T( g_0, \e, \Omega) 
\\=(1+o(1)) \cdot \hat\nu_{g_0(o)}^+((g_0 N^+_\e) v_o^+) 
\hat\nu_{g_0(o)}^-((g_0 N^-_\e) v_o^-) \op{Vol}(\Omega)\delta^{-1} e^{\delta T}
\\ =(1+O(\e))(2\e)^{-1} \delta^{-1}   e^{\delta T} \cdot ( \tilde m^{\BMS}
(\fG(g_0, \e)\otimes \Omega) +o(1)),\end{multline*}
implying  the first claim. The second claim follows from Theorem
\ref{ecount}, and Remark \ref{rem1}.  

\end{proof}

\section{ Asymptotic distribution of closed geodesics with holonomies}
\label{sec:count-equi-1}
We keep the notations $G, \Gamma, X, K, o, v_o$ etc. from section \ref{sec:closing-lemma}.
In particular, $\G$ is Zariski dense and $|m^{\BMS}|<\infty$,
$X=\Gamma\ba G/K$, and $\T^1(X)=\Gamma\ba G/M$.
In this section, we will describe the distribution of all closed geodesics of
length at most $T$ coupled together with their holonomy classes, using the results proved in section \ref{sc}.
The main ingredient is the comparison lemma \ref{comp},
which we obtain using the effective closing lemma \ref{ecl}.

\medskip

By a (primitive) closed geodesic $C$ in $\T^1(X)$, we mean a compact set of the form
 $$\G\ba \G gAM/M=\G\ba \G g A(v_o)$$
 for some $g\in G$. The length of a closed geodesic  $C=\G\ba \G gAM/M$ is same as the co-volume of 
$AM\cap g^{-1}\G g $ in $AM$. 
If we denote by $\gamma_C$ a generator of $\G\cap gAMg^{-1}$ and denote by $[\gamma_C]$ its conjugacy class
in $\G$,
then the map $$C\mapsto [\gamma_C]$$ is a bijection between 
between the set of all (primitive) closed geodesics and the set of all primitive hyperbolic conjugacy classes of $\G$.

 For each closed geodesic $C$, we denote by
  $\mathcal{L}_C$ the length measure on  $C$ and by  $h_C$
 the unique $M$-conjugacy class associated to the holonomy class of $C$.
For a primitive hyperbolic element $\gamma\in \G$, we denote by $\ell(\gamma)$
its translation length, or equivalently the length of the closed geodesic corresponding to $[\gamma]$.

\medskip

Let $M^{\textsc{c}}$ denote the space of conjugacy classes of $M$. It is known that
 $M^{\textsc{c}}$ can be identified with $\text{Lie}(S)/W$ where $S$ is a maximal torus of $M$ and
 $W$ is the Weyl group relative to $S$.
For $T>0$, define $$\mathcal G_\G(T):=\{C: C\text{ is a closed geodesic in $\T^1(X)$},\;\; \ell(C)\le T\}.$$

For each $T>0$, we define
 the measure $\mu_T$ on the product space $(\G\ba G/M) \times M^{\textsc{c}}$:
  for $f\in C(\Gamma\ba G/M)$ and any class function $\xi\in C(M)$,
$$\mu_T(f\otimes \xi)=\sum_{C\in  \mathcal G_\G(T)} \mathcal{L}_C (f) \xi(h_C).$$

We also define a measure  $\eta_T$  by
$$\eta_T(f\otimes \xi)=\sum_{C\in  \mathcal G_\G(T)} \mathcal{D}_C (f) \xi(h_C),$$
 where 
$ \mathcal{D}_C (f)=\ell(C)^{-1} \mathcal{L}_C (f)$.
If $B$ is a subset of $\G\ba G/M$ and $\Omega$ is a subset of $M$,
then we put $\mu_T(B\otimes \Omega):=\mu_T(1_B\otimes 1_\Omega)$ and
$\eta_T(B\otimes \Omega):=\eta_T(1_B\otimes 1_\Omega)$.
 \medskip

The main goal of this section is to prove the following: 
\begin{thm}\label{eqt} Let $\G$ be geometrically finite and Zariski dense.
For any bounded $f\in C(\G\ba G/M)$ and $\xi \in\op{Cl}(M)$,
we have, as $T\to \infty$, \be\label{mue2}
\mu_{T }( f\otimes \xi)
\sim  \frac{e^{\delta T} }{\delta |m^{\BMS}|} \cdot  m^{\BMS}(f\otimes \xi) ;\ee
and \be\label{mue3}
 \eta_{T }( f\otimes \xi)
\sim \frac{ e^{\delta T}  }{\delta T \cdot |m^{\BMS}|}  \cdot  m^{\BMS}(f\otimes \xi) .\ee
Moreover if $G$ and $\G$ are as Theorem \ref{m2}, then \eqref{mue2}   holds
 with an exponential error term $O(e^{(\delta-\e_1)T})$ for
some $\e_1>0$ with the implied constants depending only on the Sobolev norms of $f$ and $\xi$, and
 for some $\e_2>0$, we have
 \be\label{mue4}
 \eta_{T }( f\otimes \xi)
=\op{li} { (e^{\delta T} ) } \frac{  m^{\BMS}(f\otimes \xi)}{ |m^{\BMS}|} +O(e^{(\delta-\e_2 )T})  .\ee

\end{thm}
Theorems \ref{m1} and \ref{m2} in the introduction follow immediately from Theorem \ref{eqt},
whose proof occupies the rest of this section.

Fix a Borel subset $\Omega$ of $M^{\textsc{c}}$ and $g_0\in G$.
Recall the flow box $ \mathfrak{B}(g_0, \e)=g_0 (N^+_\e N^-\cap N^-_\e N^+AM) M_\e A_\e $ and
the notation $\tilde \fG(g_0, \e)=\tilde \pi (\fG(g_0, \e))$ from \eqref{fgdd} and \eqref{tfb}.
We will first investigate the measure $\mu_T$ restricted to the set
$\tilde \fG(g_0, \e)\otimes \Omega$.
The main idea is to relate the measure $\mu_{T}(\tilde \fG(g_0, \e )\otimes \Omega)$ with
the cardinality $\#\G\cap \V_T(g_0,\e, \Omega)$.


\medskip We fix $g_0\in \op{supp}(\tilde m^{\BMS})$ and $\e>0$ (smaller than the injectivity radius of $g_0$) from now on until Theorem \ref{bef}.
For a closed geodesic $C=\G\ba \G gA v_o \subset \G\ba G/M$,
we choose  a complete geodesic $\tilde C\subset G/M$, which is a lift of $C$.
The stabilizer $\G_{\tilde C}=\{\gamma\in \G:
\gamma(\tilde C)=\tilde C\}$ is $gAMg^{-1}\cap \G$ which is generated by a primitive hyperbolic element of $\G$, and
$C$ can be identified with $\G_{\tilde C}\ba \tilde C$.
Set \be \label{ic} I(C)=\{[\sigma ]\in \G/\G_{\tilde C} : \sigma \tilde C \cap \fG(g_0, \e) v_o\ne\emptyset\},\ee
that is, $I(C)=\{\sigma\tilde C: \sigma \tilde C \cap \fG(g_0, \e) v_o\ne\emptyset\}$.
Clearly $\#I(C)$ does not depend on the choice of $\tilde C$.
\begin{lem}\label{prelc} 
\begin{enumerate}
\item For any closed geodesic $C\subset \T^1(X)$, we have
$${\mathcal L}_C ( \tilde \fG(g_0,\e))=2\e \cdot \# I(C);$$
\item For any $T>0$,
we have \be\label{pr1} \mu_{T}(\tilde  \fG(g_0, \e)\otimes \Omega)
=2\e \cdot \sum_{ C\in \mathcal G_\G(T)}\# I(C) \cdot 1_\Omega (h_C) 
\ee  where $h_C$ is
the holonomy class about $C$.
\end{enumerate}
\end{lem}
\begin{proof} (2) immediately follows from (1). To see (1),
let $C=\G\ba \G gAv_o$. We may assume $\tilde C=gAv_o.$
We have 
\begin{align*}
{\mathcal L}_C ( { \tilde \fG}(g_0,\e)) &=
\int_{[ga_tv_o]\in \G_{\tilde C}\ba \tilde C } \sum_{\sigma \in \G } 1_{\fG(g_0,\e)} (\sg g a_t v_o) dt \\&=
\sum_{[\sg ] \in \G/\G_{\tilde C}} \int_{ga_tv_o\in \tilde C} 1_{\fG(g_0,\e)} (\sg g a_t v_o) dt .\end{align*}
 By Lemma \ref{box},
 we have 
	 \be \int_{\tilde C}1_{\fG(g_0,\e)} (\sg g a_t v_o) dt =  \begin{cases} & 
 2\e,\;\; \text{ if $\sg \tilde C\cap \fG(g_0,\e)v_o\ne 0$}
 \\  &0,\;\;  \text{  otherwise.}\end{cases} \ee
 Therefore the claim follows.
\end{proof}

Set $$\mathcal W(g_0, \e, \Omega):=
\{ gam g^{-1}: g\in \fG(g_0, \e), am\in A\Omega \} .$$
By definition, the set $\mathcal W(g_0, \e, \Omega)$ consists of hyperbolic elements.
For $T>1$, we set $$\mathcal W_T(g_0, \e, \Omega):=
\{ gam g^{-1}: g\in \fG(g_0, \e), am\in A_T^+\Omega \} . $$

We denote by $\G_h$ the set of hyperbolic elements and
by $\G_{ph}$ the set of primitive hyperbolic elements of $\G$.

\begin{prop}\label{upper}
For 
all large $T\gg 1$, we have
$$
\mu_{T}(\tilde  \fG(g_0, \e)\otimes 1_\Omega)
 =2\e\cdot \# \G_{ph}\cap 
\mathcal W_{T}(g_0,\e, \Omega)
$$
\end{prop}
\begin{proof}
We use Lemma \ref{prelc} (2):
$$ \mu_{T}(\tilde  \fG(g_0, \e)\otimes 1_\Omega)
=2\e \cdot \sum_{ C\in \mathcal G_\G(T)}\# I(C) \cdot 1_\Omega (h_C) .$$
with $I(C)=\{\sigma(\tilde C): \sigma \tilde C \cap \fG(g_0, \e)v_o \ne\emptyset\}$.

\noindent{\bf Upper bound:} Let $C\in \mathcal G_\G(T)$ be with $I(C)$ non-empty and $h_C\in \Omega$.
Without loss of generality, we may assume $\tilde C\cap  \fG(g_0, \e) v_o\ne\emptyset$.
Choose a primitive hyperbolic element
$\gamma:=\gamma_C\in \G_{\tilde C}$. 
We claim that for any $[\sigma] \in I(C)$,
 \be\label{s1}\sigma_\gamma:=\sigma \gamma \sigma^{-1}\in \W_{T}(g_0,\e, \Omega);\ee
note that $\sigma_\gamma$ is well-defined independent of the choice of a representative $\sigma$ since
$\G_{\tilde C}$ is commutative. 
Since $\tilde C\cap  \fG(g_0, \e) v_o\ne\emptyset$,
 there exists $g_1\in \fG(g_0,\e)$ such that 
$g_1v_o\in \tilde C$, and 
$\gamma = g_1 a_\gamma m_\gamma g_1^{-1}$ where $d(a_\gamma, e)=\ell(C)\le T$ and $[m_\gamma]\in \Omega$.
If $[\sigma]\in I(C)$, 
 then there exist $g_2\in \fG(g_0,\e)$ and $a_sm \in AM$ 
such that 
$$\sigma g_1a_s m = g_2.$$
Therefore, we have \[
g_2 a_{\g} {m}^{-1} m_{\g} m =\sigma{\g}\sigma^{-1}g_2 
\]
and $$\sigma_\gamma= g_2 a_\gamma m^{-1} m_{\g}m  g_2^{-1}\in \mathcal W_T(g_0, \e, \Omega).$$
proving \eqref{s1}. 

To see that the map $[\sigma] \mapsto \sigma_\gamma$ is injective on $I(C)$, it suffices
to recall that the centralizer of $\gamma$ in $\G$ is $\G_{\tilde C}$. Hence this proves the upper bound.

\noindent{\bf Lower bound:} We write
$$\# \G_{ph}\cap \mathcal W_{T}(g_0,\e, \Omega)
=\sum \# [\gamma]\cap \mathcal W_{T}(g_0,\e, \Omega)$$
where the sum ranges over the conjugacy classes 
$$[\gamma]=\{\gamma_0\in \Gamma_{ph}:\gamma_0
\text{ is conjugate to $\gamma$ by an element of $\Gamma$} \}$$ of primitive hyperbolic elements of $\Gamma$.
Fix a primitive hyperbolic clement $\gamma \in \mathcal W_{T}(g_0,\e, \Omega)$.
So
 there exists $g\in \fG(g_0,\e )$ such that
$\gamma =g a_\gamma m_\gamma g^{-1}$ with $ a_\gamma\in A_T^+$ and
$[m_\gamma]\in \Omega$.
Let $C=\G\ba \G g aAv_o$ and $\tilde C=gAv_o$. Then
the length of $C$ is at most $T$. 

For each element
$\sigma':=\sigma \gamma \sigma^{-1} \in [\gamma]\cap \mathcal W_{T}(g_0,\e, \Omega)$, 
we have $\sigma \gamma \sigma^{-1}=g_2 a_\gamma m g_2^{-1}$ for some $[m]\in \Omega$ and $g_2\in \fG(g_0, \e)$.

Since  $\sigma^{-1} g_2Av_o$
is the oriented axis for $\gamma$,
$\sigma^{-1}g_2Av_o=\tilde C$ by Corollary \ref{cuq}. Therefore $g_2 v_o\in \sigma(\tilde C)\cap \fG(g_0, \e)v_o$, and hence
$\sigma(\tilde C)\in I(C)$. Since the map $\sigma' =\sigma \gamma\sigma^{-1} \mapsto \sigma (\tilde C)$
is well-defined and injective,
this proves the lower bound by \eqref{pr1}.\end{proof}
\bigskip
Indeed the proof of Proposition \ref{upper} gives that if $C$ is a closed geodesic and
$[\gamma_C]$ is the conjugacy class of primitive hyperbolic elements which corresponds to $C$, then
\be \mathcal{L}_C (\tilde \fG(g_0,\e) ) 1_\Omega(h_C)= 2\e\cdot  \# I(C) \cdot 1_\Omega(h_C)= 2\e\cdot
 \# [\gamma_C]\cap \mathcal W (g_0,\e, \Omega).\ee

Recall the notation $$\mathcal V_T(g_0,\e, \Omega):=
\fG(g_0, \e  ) A_T^+ \Omega  \fG(g_0, \e )^{-1}.$$ 
 Let $c>1$ be a fixed upper bound for all implied
constants  involved in the $O$ symbol 
in Lemma \ref{ecl} and the constant in \eqref{comp'}.

The effective closing lemma \ref{eecl} implies that for all large $T\gg T_0$,
\be \label{mewl} \mathcal V_{T}(g_0,\e(1-ce^{-T/2}), \Omega_{c\e}^-)- \mathcal V_{T_0}(g_0,\e, \Omega)
\subset \mathcal W_{T}(g_0,\e, \Omega) .\ee


\begin{lem}\label{upper2} 
For $T\gg 1$, we have
$$ \# \G\cap\left( \mathcal W_{T}(g_0,\e, \Omega) -
\mathcal W_{2T/3}(g_0,\e, \Omega) \right)\le 
  \# \G_{ph}\cap 
\mathcal W_{T}(g_0,\e, \Omega).
$$ \end{lem}
\begin{proof} Note that,
 since $ \W_{T}(g_0,\e, \Omega)$
consists of hyperbolic elements,
 $$
\#\G_{ph} \cap \W_{T}(g_0,\e, \Omega) \\ = \# \G \cap \W_{T}(g_0,\e, \Omega)
 -\#   ( \cup_{k\ge 2}\G_{ph}^k) 
 \cap \W_{T}(g_0,\e, \Omega)  $$ where $\Gamma_{ph}^k=\{\sigma^k:\sigma\in \G_{ph}\}$.

On the other hand, by Lemma \ref{bef3} and \eqref{mewl}, for some constant $c_0>1$,
$$c_0^{-1}e^{\delta T} \le \#\G\cap \W_{T}(g_0,\e, \Omega)\le c_0 e^{\delta T}$$
for all $T\gg 1$.

Hence for $T\gg 1$.
\begin{multline*} \#   ( \cup_{k\ge 2}\G_{ph}^k) 
 \cap \W_{T}(g_0,\e, \Omega) \le \sum_{k\ge 2}\# \G\cap \W_{T/k}(g_0,\e, \Omega)
\\ \le c_0 \sum_{k\ge 2} e^{\delta T/k} \le  \#\G\cap \W_{2T/3}(g_0,\e, \Omega)
\end{multline*}
proving the claim. \end{proof}



By the ergodicity of the geodesic flow with respect to the BMS measure on $\G\ba G$ \cite{Wi}, for any $g_0\in \text{supp}(m^{\BMS})$,
a random $AM$-orbit in $\G \ba G$ comes back  to the flow box
$\mathfrak B(g_0,\e)$  infinitely often.  The effective closing lemma implies that that there is an arbitrarily long closed geodesic
nearby whose holonomy class is $O(\e)$-close to the $M$-component of $g_0$ in  the $N^+N^-AM$ decomposition.
Since the projection  of $\text{supp}(m^{\BMS})$ to the $M$-components is all of $M$,
this shows not only the existence of a closed geodesic but also the density of holonomy classes in the space of all conjugacy classes of $M$.

 The comparison lemma below gives a much stronger control
on the number of closed geodesics whose holonomy classes contained in a fixed subset of $M$
in terms of lattice points, whose cardinality is controlled by the mixing.  

\medskip


\medskip 

\begin{lem}[Comparison Lemma] \label{comp}
For all $T\gg 1$, we have
\begin{multline*} 2\e\cdot \# \G\cap \left(\mathcal V_{T}(g_0,\e(1-ce^{-T/2}), \Omega_{c\e}^-)-\mathcal V_{2T/3}(g_0,\e, \Omega)\right)
  \\  \le
 \mu_{T}(\tilde  \fG(g_0,\e)\otimes \Omega) \le 2\e\cdot \# \G\cap \mathcal V_{T}(g_0,\e, \Omega).
\end{multline*}  where $\Omega_{c\e}^-=\cap_{m_i\in M_{c\e }} m_1 \Omega m_2 $.
\end{lem}
\begin{proof} The upper bound is immediate from the definition of the sets and Proposition \ref{upper}.

Proposition \ref{upper}, Lemma \ref{upper2} and \eqref{mewl} imply the lower bound.
\end{proof}

\begin{thm}\label{bef}
We have \be
 \mu_{T}(\tilde  \fG(g_0, \e)\otimes \Omega)
=(1+O(\e))\frac{ e^{\delta T}}{\delta \cdot |m^{\BMS}|} \cdot ( m^{\BMS}
( \tilde \fG(g_0, \e)\otimes \Omega) +o(1)) \ee
where the implied constants are independent of $g_0$ and $\e$. 

Moreover if $G$ and $\G$ are as in Theorem \ref{m2}, $o(1)$ can be replaced by $O(e^{-\e_1 T})$ for some positive $\e_1>0$.
\end{thm}
\begin{proof}  This follows from the comparison lemma \ref{comp} and
Theorem \ref{bef3}. \end{proof}

We note that we do not require $\G$ to be geometrically finite in the following theorem.
\begin{thm}\label{eqthmc} Let $\G$ be Zariski dense with $|m^{\BMS}|<\infty$.
For any $f\in C_c(\G\ba G/M)$ and $\xi \in\op{Cl}(M)$,
we have, as $T\to \infty$, \be\label{mue}
 \mu_{T }( f\otimes \xi)
\sim  \frac{e^{\delta T} \cdot m^{\BMS}(f\otimes \xi)}{\delta \cdot |m^{\BMS}|}. \ee
Moreover if $G$ and $\G$ are as in Theorem \ref{m2}, then $\eqref{mue}$ holds with an exponential error term $O(e^{-\e_1t})$ for
some $\e_1>0$ with the implied constants depending on the Sobolev norms of $f$ and $\xi$.
\end{thm}
\begin{proof} We normalize $|m^{\BMS}|=1$. Using a partition of unity argument,
we can assume without loss of generality that $f$ is supported on
  $\tilde  \fG(g_0, \e)$ for some $g_0\in\op{supp}(\tilde m^{\BMS})$ and $\e>0$.
  Now for arbitrarily small $0<\rho<\e$, we can approximate $f$ as step functions
  which are linear combination of characteristic functions
  of $\tilde  \fG(h, \rho)$'s with $h\in  \tilde  \fG(g_0, \e)$.
  Now applying Proposition \ref{bef} to each
  $1_{\tilde  \fG(h, \rho)}\otimes 1_\Omega$, 
  we deduce that  \begin{multline*}
 (1- c \rho)   m^{\BMS}(f\otimes \xi) \le \liminf_T e^{-\delta T}  \mu_{T}( f\otimes \xi)
\le \\ 
\limsup_T \delta e^{-\delta T}  \mu_{T}( f\otimes \xi) \le (1+c\rho)   m^{\BMS}(f\otimes \xi) \end{multline*}
Since $\rho>0$ is arbitrary, this implies the claim when $\xi$ is the characteristic function of $\Omega$
whose boundary has a measure zero. Via the identification $M^{\textsc{c}}=\op{Lie}(S)/W$ where
$S$ is a maximal torus of $M$ and $W$ is the Weyl group relative to $S$,
extending the above claim from characteristic (class) functions to continuous (class)
functions is similar to the above arguments. This establishes \eqref{mue}.
When the effective version of Theorem \ref{bef} holds, we also obtain an error term in this argument.
\end{proof}

\noindent {\bf Contribution of the cusp and equidistribution
for bounded functions}\label{sec:cusp}
 In order to extend Theorem \ref{eqthmc} to bounded continuous functions, which are not necessarily compactly supported,
 we now assume that $\G$ is geometrically finite and use the following theorem of Roblin \cite{R_T} (Theorem \ref{cc}).

We denote by  $\mathcal C(\G)$ the convex core of $\G$.
Let $\e_0>0$ be the Margulis constant for $\G$. Then
$\{x\in \mathcal C(\G):\text{injectivity radius at } x \ge \e_0\}$
is called the thick part and its complement is called the thin part. We will denote them $\mathcal C(\G)_{thick}$
and $\mathcal C(\G)_{thin}$ respectively.
When $\G$ is a geometrically finite group,
 the thin part part consists of finitely many disjoint cuspidal regions (called horoballs), say, $\mathcal H_1, \cdots, \mathcal H_k$
 based at parabolic fixed points $p_1, \cdots, p_k$ respectively.
 We denote by $\G_{p_i}$ the stabilizer of $p_i$ in $\G$. Also, fixing $o$ in the thick part of $\mathcal C(\G)$, let $q_i$
 denote the point of intersection between the geodesic ray connecting $o$ and $p_i$ with the boundary of the horoball $\mathcal H_i$.

\begin{prop}\cite{DOP} \label{cusp}
If $\G$ is geometrically finite, then for each parabolic fixed point $p_i\in \Lambda(\G)$, we have
$$ \sum_{\sigma\in\G_{p_i}} d(q_i,\sigma q_i) e^{- \delta \cdot d(q_i,\sigma q_i)}<\infty .$$
\end{prop}

  For any $r\geq0$ denote by $\hcal_i( r)$ the horoball contained in $\hcal_i$ 
whose boundary is of distance $r$ to $\partial\hcal_i.$
Put $\cusp( r)=\cup_i\hcal_i( r).$

\begin{thm}[Roblin, \cite{R_T}]\label{cc} There exist absolute constants $c_0, c_1>0$ such that for any $T\gg 1$,
$${e^{-\delta T} }\cdot {\mu_T( \cusp (r) K)}\le c_1\sum_{i=1}^k
\sum_{\sigma\in\G_{p_i}, d(q_i,\sigma q_i) >2r - c_0} ( d(q_i,\sigma q_i) -2r + c_0)e^{- \delta \cdot d(q_i,\sigma q_i)}.
$$

In particular, if $G=\SO(n,1)^\circ$, then 
\be\label{ef} {e^{-\delta T} }\cdot {\mu_T( \cusp (r) K)}\ll e^{(\kappa-2\delta)r}\ee
where $\kappa=\max{\rm rank}(p_i)$.
 \end{thm}

These estimates and the proof for compactly supported functions 
imply the result for bounded functions.

\noindent{\bf Proof of Theorem \ref{eqt}:} We may assume $ |m^{\BMS}|=1$. By Proposition \ref{cusp},
$$\sum_{\sigma\in\G_{p_i}, d(q_i, \sigma q_i) >s} d(q_i,\sigma q_i) e^{- \delta \cdot d(q_i,\sigma q_i)} \to 0$$
as $s\to \infty$.
 Therefore by Theorem \ref{cc},
$${e^{-\delta T} }\cdot {\mu_T( \cusp (r) K)} =o_r(1) .$$
If we denote by $\Phi_r $ a continuous approximation
of the unit neighborhood of $\mathcal C(\G) -(\cup
\mathcal H_i(r))$ (that is, $\Phi_r=1$ on the neighborhood and
 $0$ outside a slightly bigger neighborhood) then Theorem \ref{eqthmc} implies that
$$e^{-\delta T}\delta \mu_{T }( f\cdot \Phi_r \otimes \xi) =
  m^{\BMS}(f\cdot \Phi_r \otimes \xi)+o_T(1).$$
Hence
$$\left|e^{-\delta T}\delta  \mu_{T }( f \otimes \xi) -m^{\BMS}(f \otimes \xi)\right|
=o_T(1)+o_r(1) +m^{\BMS} (\cusp (r) K)$$
since the support of $m^{\BMS}$ is contained in $\mathcal C(\G)$.
By taking $r\to \infty$, we finish the proof of the first claim \eqref{mue2}. In view of \eqref{ef} and Theorem \ref{eqthmc},
the claim on the error term follows as well.

 We now deduce \eqref{mue3} from \eqref{mue2} ; this
is done in~\cite[Section 5]{R_T}; we recall the proof for 
the convenience of the reader.

Without loss of generality we may  assume $f\otimes\xi\geq 0$.
It follows from the definition that 
\be\label{eq:upper-nlength}
\delta Te^{-\delta T}\eta_T(f\otimes \xi)\geq \delta e^{-\delta T}\mu_T(f\otimes \xi).
\ee
Therefore \eqref{mue2} implies that
$$\liminf_T \delta Te^{-\delta T}\eta_T(f\otimes \xi)\ge  m^{\BMS}(f\otimes \xi). $$
We now bound $\eta(f\otimes \xi)$ from above. Let $\vare>0$ be
small fixed number. We have
\begin{align}
&\label{eq:lower-nlength}\delta Te^{-\delta T}\eta_T(f\otimes \xi)=\\
\notag&\delta Te^{-\delta T}\left(\sum_{ \mathcal G_\G((1-\vare)T)} \mathcal{D}_C (f) \xi(h_C)+\sum_{\mathcal G_\G(T)\setminus \mathcal G_\G((1-\vare)T)} \mathcal{D}_C (f) \xi(h_C)\right)\leq\\
\notag&\delta Te^{-\delta T}\left(\sum_{ \mathcal G_\G((1-\vare)T)} \ell(C)\mathcal{D}_C (f) \xi(h_C)+\sum_{ \mathcal G_\G(T)\setminus \mathcal G_\G((1-\vare)T)} \tfrac{\ell(C)}{(1-\vare)T}\mathcal{D}_C (f) \xi(h_C)\right)\leq\\
\notag&Te^{-\delta\vare T}\left(\delta e^{-\delta((1-\vare)T)}\mu_{(1-\vare)T}(f\otimes\xi)\right)+\frac{\delta e^{-\delta T}}{1-\vare}\left(\mu_T(f\otimes\xi)-\mu_{(1-\vare)T}(f\otimes\xi)\right).
\end{align}
Therefore  again by Theorem~\ref{eqt},
\begin{multline*}\limsup_T \delta Te^{-\delta T}\eta_T(f\otimes \xi)\le  m^{\BMS}(f\otimes \xi)
\left(   \limsup_T  Te^{-\delta\vare T}  + \frac{1}{1-\e} + e^{-\e \delta T} \right)\\
\le \frac{1}{1-\e}m^{\BMS}(f\otimes \xi)
\end{multline*}
Since $\e>0$ is arbitrary, this proves the claim. 

When \eqref{mue2} is effective, we use Abel's summation formula to deduce~\eqref{mue4}.
Given functions $f$ and $\xi$ and $T>1$, we note that the map $\alpha(t):=\mu_t(f\otimes\xi)$ defines a step function on $(0,T]$
with finitely many jumps at values of $t$ where there is a closed geodesic of length $t.$ 
The amount of jump at $t$ is given by $\sum_{C:\ell(C)=t}\mathcal L_C(f)\xi(h_C).$ Let $\varphi(t)=1/t,$ we then 
compute the Riemann-Stieltjes integral  
\[
\int_0^T\varphi(t)d\alpha=\sum_{0<t\leq T}\sum_{C:\ell(C)=t}\frac{1}{t}\mathcal L_C(f)\xi(h_C)=\eta_T(f\otimes\xi).
\]
Let $t_0>0$ be the length of the shortest geodesic. 
Using integration by parts 
and the effective estimate for~\eqref{mue2} we get the following: for simplicity,  we write $c_{f\otimes \xi}:=\tfrac{  m^{\BMS}(f\otimes \xi)}{ |m^{\BMS}|}$.
\begin{align*}
\eta_T(f\otimes\xi)&=\frac{\mu_T(f\otimes\xi)}T -\frac{\mu_{t_0}(f\otimes\xi)}{t_0}  -\int_{t_0}^T\mu_t(f\otimes\xi)\varphi'(t)dt\\
&=c_{f\otimes \xi} \frac{e^{\delta T} }{\delta T}  +\int_{t_0}^T\frac{\mu_t(f\otimes\xi)}{t^2}dt+O(e^{(\delta-\e_1)T})
\\ &= c_{f\otimes \xi} \left( \frac{e^{\delta T} }{\delta T} +\int_{t_0}^T\frac{e^{\delta t} }{\delta t^2}dt\right) +O(e^{(\delta-\e_2)T})
\end{align*}
for some $\e_2>0$.
Since $\op{li}(e^{\delta T})=\int_{2}^{e^{\delta T}}\frac{dt}{\log t} =\frac{e^{\delta T} }{\delta T} -\frac{e^{\delta 2} }{2\delta }  +
\int_{(\log 2)/\delta}^{T}\frac{e^{\delta s}}{\delta s^2} ds$,
we obtain
$$\eta_T(f\otimes\xi)=  c_{f\otimes \xi} \op{li}(e^{\delta T}) +O(e^{(\delta-\e_2)T}),$$
completing the proof.

\end{document}